\documentclass{amsart}
\usepackage[utf8]{inputenc}
\usepackage{amsthm}
\usepackage{amsmath}
\DeclareMathOperator\arctanh{arctanh}
\DeclareMathOperator\dist{dist}
\DeclareMathOperator\supp{supp}
\usepackage{enumerate}   
\usepackage{amssymb}
\usepackage{fancyhdr} 
\usepackage[left=3.5cm,right=3.5cm,top=3.5cm,bottom=3.5cm]{geometry}
\usepackage{xcolor}
\newcommand{\norm}[1]{\left\lVert#1\right\rVert}
\usepackage{graphicx}
\usepackage{ esint }
\newtheorem{theorem}{Theorem}[section]
\newtheorem{proposition}[theorem]{Proposition}
\newtheorem{corollary}[theorem]{Corollary}
\newtheorem{lemma}[theorem]{Lemma}
\theoremstyle{definition}
\newtheorem{example}[theorem]{Example}
\theoremstyle{remark}
\newtheorem*{remark}{Remark}
\renewcommand\Re{\operatorname{Re}}
\renewcommand\Im{\operatorname{Im}}

\usepackage{url}
\usepackage[hidelinks]{hyperref}
\hypersetup{
	colorlinks=true,
	linkcolor=cyan,
	filecolor=cyan,
	citecolor =cyan,      
	urlcolor=cyan,
}
\author{Athanasios Kouroupis}
\address{Department of Mathematical Sciences, Norwegian University of Science and Technology (NTNU), 7491 Trondheim, Norway}
\email{athanasios.kouroupis@ntnu.no}
\author{Karl-Mikael Perfekt}
\address{Department of Mathematical Sciences, Norwegian University of Science and Technology (NTNU), 7491 Trondheim, Norway}
\email{karl-mikael.perfekt@ntnu.no}
\title{Composition operators on weighted Hilbert spaces of Dirichlet series}
\date{}

\begin{document}
	\begin{abstract}
		We study composition operators of characteristic zero on weighted Hilbert spaces of Dirichlet series. For this purpose we demonstrate the existence of weighted mean counting functions associated with the Dirichlet series symbol, and provide a corresponding change of variables formula for the composition operator. This leads to natural necessary conditions for the boundedness and compactness. For Bergman-type spaces, we are able to show that the compactness condition is also sufficient, by employing a Schwarz-type lemma for Dirichlet series.
	\end{abstract}
	\maketitle
	\section{Introduction}
	For $a\leq 1$ we define the weighted Hilbert space $\mathcal{D}_a$ of Dirichlet series as 
	$$\mathcal{D}_a=\left\{f(s)=\sum_{n\geq1}\frac{a_n}{n^s}:\norm{f}_a^2=|a_1|^2+\sum_{n\geq2}|a_n|^2\log(n)^a<\infty \right\}.$$
	The space $\mathcal{D}_{0}$ coincides with the Hardy space $\mathcal{H}^2$ of Dirichlet series with square summable coefficients, which was systematically studied in an influential article of Hedenmalm, Lindqvist, and Seip \cite{HLS97}. For $a < 0$ we refer to $\mathcal{D}_a$ as a Bergman space and for $a>0$ as a Dirichlet space, see \cite{MCC04}.
	
By the Cauchy--Schwarz inequality, $\mathcal{D}_a$ is a space of analytic functions in the half-plane $\mathbb{C}_{\frac{1}{2}}$, where $\mathbb{C}_\theta=\{s\in \mathbb{C}:\Re s>\theta\}$. Therefore, if $\psi:\mathbb{C}_{\frac{1}{2}}\rightarrow\mathbb{C}_{\frac{1}{2}}$ is an analytic function, the composition operator $C_\psi (f)=f\circ\psi$ defines an analytic function in $\mathbb{C}_{\frac{1}{2}}$ for every $f\in\mathcal{D}_a$. Gordon and Hedenmalm \cite{GH99} determined the class $\mathfrak{G}$ of symbols which generate bounded composition operators on the Hardy space $\mathcal{H}^2$.  The Gordon--Hedenmalm class $\mathfrak{G}$ consists of all functions $\psi(s) = c_0s+\varphi(s)$, where $c_0$ is a non-negative integer, called the characteristic of $\psi$,  and $\varphi$ is a Dirichlet series such that:
	\begin{enumerate}[(i)]
		\item If $c_0=0$, then $\varphi(\mathbb{C}_0)\subset\mathbb{C}_\frac{1}{2}$.\label{item1}
		\item If $c_0\geq 1$, then $\varphi(\mathbb{C}_0)\subset\mathbb{C}_0$ or $\varphi\equiv i\tau$ for some $\tau\in\mathbb{R}$.\label{item2}
	\end{enumerate}
We will use the notation $\mathfrak{G}_0$ and $\mathfrak{G}_{\geq 1}$ for the subclasses of symbols that satisfy (\ref{item1}) and (\ref{item2}), respectively. In either case, the mapping properties of $\varphi$ and Bohr's theorem imply that the Dirichlet series $\varphi$ necessarily has abscissa of uniform convergence $\sigma_u(\varphi)\leq0$, see \cite[Theorem 8.4.1]{QQ20}.
	
	By what is essentially the original argument of Gordon and Hedenmalm, the condition that $\psi \in\mathfrak{G}$ is necessary for a composition operator $C_\psi : \mathcal{D}_a\rightarrow\mathcal{D}_a$ to be bounded. In the Bergman case $a < 0$, this is also known to be sufficient \cite{Bail15, BB16}. When $\psi \in \mathfrak{G}_0$, the proof of boundedness of $C_\psi : \mathcal{D}_a\rightarrow\mathcal{D}_a$, $a < 0$, due to Bailleul and Brevig \cite{BB16}, has a rather serendipitous flavor. In Section 3 we will supply a more systematic proof based on a Schwarz lemma for Dirichlet series, Lemma \ref{3.4}.
	
	Beyond this, we will focus on composition operators induced by symbols $\varphi \in \mathfrak{G}_0$. The compact operators $C_\varphi \colon \mathcal{H}^2 \to \mathcal{H}^2$ were characterized only very recently in \cite{BP21}, in terms of the behavior of the mean counting function
	$$M_{\varphi, 1}(w) = \lim_{\sigma \to 0^+}\lim_{T \to \infty}  \frac{\pi}{T}\sum\limits_{\substack{s\in\varphi^{-1}(\{w\})\\
			|\Im s|<T\\
			\sigma<\Re s<\infty}} \Re s ,\qquad w\neq\varphi(+\infty).$$
	The main purpose of this article is to explore analogous tools and results in the weighted setting.
	
	From Carlson's theorem \cite[Lemma 3.2]{HLS97} one deduces the following formula of Littlewood--Paley type,
	\begin{equation}\label{1}
	\norm{f}_a^2=|f(+\infty)|^2+\frac{2^{1-a}}{\Gamma(2-a)}\lim\limits_{\sigma_0\rightarrow0^+}\lim\limits_{T\rightarrow \infty}\frac{1}{T}\int\limits_{\sigma_0}^{\infty}\int\limits_{-T}^{T}\left|f'(\sigma+it)\right|^2\sigma^{1-a} \, dt \, d\sigma,
	\end{equation}
valid for $f\in \mathcal{D}_a$ such that $\sigma_u(f)\leq0$. From this point of view, the space $\mathcal{D}_a$ is analogous to the weighted Hilbert space $D_\alpha$, consisting of those holomorphic functions $g$ on the unit disk such that
\begin{equation}\label{eq:disknorm}
\norm{g}_{D_\alpha}^2=|g(0)|^2+\int\limits_\mathbb{D}|g'(z)|^2(1-|z|^2)^\alpha \, dA(z)<\infty,
\end{equation}
	where $\alpha=1-a\geq0$ and $dA(z)=dx \, dy,\,z=x+iy$. By the results of \cite{KL12, PP13, SHAP87}, a holomorphic self-map of the unit disk $\phi \colon \mathbb{D}\to \mathbb{D}$ induces a compact composition operator on $D_\alpha,$ $\alpha>0$, if and only if  
	\begin{equation}
	\lim\limits_{|z|\rightarrow1^-}\frac{N_{\phi,\alpha}(z)}{(1-|z|^2)^\alpha}=0,
	\end{equation}
	where for $\alpha=1$, $N_{\phi,1}$ is the classical Nevanlinna counting function $$N_{\phi}(z)=N_{\phi,1}(z)=\sum\limits_{z_i\in \phi^{-1}(\{z\})}\log\frac{1}{|z_i|}, \qquad z \neq \phi(0),$$ 
	 and for $\alpha \neq 1$, $N_{\phi,\alpha}$ is the generalized Nevanlinna counting function
	$$N_{\phi,\alpha}(z)=\sum\limits_{z_i\in \phi^{-1}(\{z\})}(1-|z_i|^2)^\alpha.$$

	A key step in the disk setting is to introduce a non-injective change of variables in \eqref{eq:disknorm}, resulting in what is known as a Stanton formula. In our setting, for $\varphi\in\mathfrak{G}_0$, making the change of variables in \eqref{1} yields that 
	\begin{equation*}
	\norm{C_\varphi(f)}_{a}^2=|f(\varphi(+\infty))|^2+\frac{2^{1-a}}{\pi\Gamma(2-a)}\lim_{\sigma\rightarrow0^+}\lim_{T\rightarrow \infty}\int\limits_{\mathbb{C}_{\frac{1}{2}}}|f'(w)|^2M_{\varphi,1-a}(w, \sigma, T) \, dA(w),
	\end{equation*}
	where
	$$M_{\varphi,a}(w, \sigma, T)=\frac{\pi}{T}\sum\limits_{\substack{s\in\varphi^{-1}(\{w\})\\
			|\Im s|<T\\
			\sigma<\Re s<\infty}}\left(\Re s\right)^a,\qquad w\neq\varphi(+\infty).$$
	For a Dirichlet series $\varphi$ with abscissa of uniform convergence $\sigma_u(\varphi)\leq 0$, we therefore introduce the weighted mean counting functions
	$$M_{\varphi,a}(w,\sigma)=\lim_{T\rightarrow \infty}\frac{\pi}{T}\sum\limits_{\substack{s\in\varphi^{-1}(\{w\})\\
			|\Im s|<T\\
			\sigma<\Re s<\infty}}\left(\Re s\right)^{a},\qquad w\neq\varphi(+\infty),$$
	and
	$$M_{\varphi,a}(w)=\lim_{\sigma\rightarrow0^+}M_{\varphi,a}(w,\sigma),$$
	if these limits exist.
	
	Jessen and Tornehave \cite[Theorem 31]{JT45} studied the unweighted counting function $M_{\varphi,0}(w,\sigma)$ in the context of Lagrange's mean motion problem. They proved that the counting function exists for $\sigma>0$ and $w\neq\varphi(+\infty)$, and that it satisfies
	$$M_{\varphi,0}(w,\sigma)=-\mathcal{J}'_{\varphi-w}(\sigma^+),$$
	where $\mathcal{J}'_{\phi-w}(\sigma^+)$ is the right-derivative of the Jessen function,
	\begin{equation} \label{eq:jessenfcn}
	\mathcal{J}_{\varphi-w}(\sigma)=\lim\limits_{T\rightarrow\infty}\frac{1}{2T}\int\limits_{-T}^T\log|\varphi(\sigma+it)-w| \, dt.
	\end{equation}
	On the basis of this and Littlewood's lemma,  it was demonstrated in \cite{BP21} that the weighted mean counting function $M_{\varphi,1}(w,\sigma)$ also exists for $\sigma > 0$ and $w\neq \varphi(+\infty)$.  Additionally, if $\varphi$ belongs to the Nevanlinna class of Dirichlet series 
	$\mathcal{N}_u$, that is, $\sigma_u(\varphi)\leq0$ and $$\limsup\limits_{\sigma\rightarrow0^+}\frac{1}{2T}\int\limits_{-T}^T\log^+|\varphi(\sigma+it)|dt<+\infty,$$
	then $$M_{\varphi,1}(w)=\lim\limits_{\sigma\rightarrow0^+}\mathcal{J}_{\varphi-w}(\sigma)-\log|\varphi(+\infty)-w|<+\infty.$$
	
	In Section~4 we will investigate the existence of the weighted mean counting functions $M_{\varphi, a}$.
	\begin{theorem}\label{1.1}
		For $a \in \mathbb{R}$, let $\varphi$ be a Dirichlet series such that $\sigma_u(\varphi)\leq 0$ and $\varphi(+\infty)\neq w$. Then the counting function $M_{\varphi,a}(w,\sigma)$ exists and is right-continuous on $\sigma>0$. Furthermore,
		\begin{equation}\label{3}
		M_{\varphi,a}(w,\sigma)=M_{\varphi,0}(w,\sigma)\sigma^{a}+a\int\limits_{\sigma}^{\infty}t^{a-1}M_{\varphi,0}(w,t) \, dt.
		\end{equation}
		For $\sigma_\infty > 0$ sufficiently large, depending on $\varphi$ and $w$, we also have that
		\begin{equation} \label{eq:meanintbyparts}
			\begin{aligned}
		M_{\varphi,a}(w,\sigma)& - M_{\varphi,0}(w,\sigma)\sigma^{a} = \\ &a\sigma^{a-1}\mathcal{J}_{\varphi-w}(\sigma)-a\sigma_\infty^{a-1}\log|\varphi(+\infty)-w|-a(1-a)\int\limits_{\sigma}^{\sigma_\infty}t^{a-2}\mathcal{J}_{\varphi-w}(t)dt.
			\end{aligned}
		\end{equation}
	\end{theorem}
In Theorem~\ref{average} we will furthermore obtain the integral representation 
$$M_{\varphi,a}(w)=\int\limits_{\mathbb{T}^\infty}M_{\varphi_\chi,a}(w,0,1) \, dm_\infty(\chi)$$
 of the weighted mean counting function, where $dm_\infty$ denotes the Haar measure on the infinite polytorus $\mathbb{T}^\infty$, and $\varphi_\chi$ denotes the Dirichlet series $\varphi$ twisted by the character $\chi \in \mathbb{T}^\infty$, see Section 2. In the case that $a \geq 1$, we are from this formula able to deduce that
 $$M_{\varphi, a}(w) = \lim_{T\rightarrow \infty}\frac{\pi}{T}\sum\limits_{\substack{s\in\varphi_\chi^{-1}(\{w\})\\
 		|\Im s|<T\\
 		\Re s>0}}\left(\Re s\right)^a$$
for almost every $\chi \in \mathbb{T}^\infty$. That is, it is almost surely possible to interchange the $T$- and $\sigma$-limits in the definition of $M_{\varphi, a}(w)$. When $a = 1$, this partially resolves \cite[Problem~1]{BP21}.

	In Section~5 we then prove the analogue of the Stanton formula.
	\begin{theorem}\label{1.2}
		Suppose that $\varphi\in \mathfrak{G}_0$ and that $a \leq 1$. Then, for every $f\in \mathcal{D}_a$,
		\begin{equation}\label{4}
		\norm{C_\varphi(f)}_a^2=|f(\varphi(+\infty))|^2+\frac{2^{1-a}}{\Gamma(2-a)\pi}\int\limits_{\mathbb{C}_{\frac{1}{2}}}|f'(w)|^2 \, M_{\varphi,1-a}(w)dA(w).
		\end{equation}
		If $a \leq 0$, then $M_{\varphi,1-a}(w)$ exists and is finite for every $w\in\mathbb{C}_{\frac{1}{2}}\setminus\{\varphi(+\infty)\}$.
	\end{theorem}
	\begin{remark}
		For $a \geq 1/2$, the mean counting function $M_{\varphi,1-a}(w) = \lim_{\sigma \to 0^+} M_{\varphi, 1-a}(w, \sigma)$ can be infinite everywhere, see Example \ref{4.7}. In particular, both sides of \eqref{4} can be infinite.  When $0 < a < 1/2$, we do not know if $M_{\varphi,1-a}(w)$ is finite for every $\varphi\in\mathfrak{G}_0$ and $\varphi(+\infty)\neq w$.
	\end{remark}

We use Theorem~\ref{1.2} to characterize the compact composition operators in the Bergman setting.
	\begin{theorem}\label{1.3}
		Let $\varphi\in \mathfrak{G}_0$. Then the induced composition operator $C_\varphi$ is compact on the Bergman space $\mathcal{D}_{-a}$, $a > 0$, if and only if 
		\begin{equation}\label{5}
		\lim_{\Re w\rightarrow\frac{1}{2}^+}\frac{M_{\varphi,1+a}(w)}{\left( \Re w-\frac{1}{2}\right)^{1+a}}=0.
		\end{equation}
	\end{theorem}
	In addition to the change of variable formula, our Schwarz-type lemma, Lemma~\ref{3.4}, is essential to proving the sufficiency of \eqref{5}. In this context, we note that Bayart \cite{BAY21} recently showed that the condition $\lim\limits_{\Re s\rightarrow0^+}\frac{\Re\varphi(s)-\frac{1}{2}}{\Re s}=\infty$ is sufficient, but not necessary, for the operator $C_\varphi \colon \mathcal{D}_{-a} \to \mathcal{D}_{-a}$ to be compact.
	
	Finally, we consider the Dirichlet-type spaces $\mathcal{D}_a$ for $0 < a < 1$.  We prove that the analogue of \eqref{5} remains necessary for the composition operator to be compact, and we give an analogous necessary condition for boundedness. In Example \ref{5.6} we observe that this condition is not sufficient for the operator to be bounded, at least not when $a \geq 1/2$. 
	
	\begin{theorem}\label{1.4}
		Suppose that $0 < a < 1$ and let $\varphi\in\mathfrak{G}_0$. If the operator $C_\varphi$ is bounded on the Dirichlet space $\mathcal{D}_{a}$, then for every $\delta>0$ there exists a constant $C(\delta)>0$ such that 
		\begin{equation}\label{6}
		\frac{M_{\varphi,1-a}(w)}{\left(\Re w-\frac{1}{2}\right)^{1-a}}< C(\delta), \qquad w\in \mathbb{C}_{\frac{1}{2}}\setminus D(\varphi(+\infty),\delta).
		\end{equation}
		If $C_\varphi:\mathcal{D}_{a}\rightarrow\mathcal{D}_{a}$ is compact, then
\begin{equation}\label{7}
\lim_{\Re w\rightarrow\frac{1}{2}^+}\frac{M_{\varphi,1-a}(w)}{\left(\Re w-\frac{1}{2}\right)^{1-a}}=0.
\end{equation}
	\end{theorem}
	In the special case where the symbol $\varphi$ has bounded imaginary parts and the associated counting function is locally integrable, we can also prove that \eqref{6} is sufficient for the composition operator $C_\varphi$ to be bounded, and that \eqref{7} is sufficient for a bounded composition operator $C_\varphi$ to be compact.
	
	\subsection*{Notation}
	Throughout the article, we will employ the convention that $C$ denotes a positive constant which may vary from line to line. When we wish to clarify that the constant depends on some parameter $P$, we will write that $C = C(P)$. Furthermore, if $A = A(P)$ and $B = B(P)$ are two quantities depending on $P$, we write $A \approx B$ to signify that there are constants $c_1, c_2 > 0$ such that $c_1 B \leq A \leq c_2B$ for all relevant choices of $P$.
	
	\subsection*{Acknowledgments}  
	We thank Ole Fredrik Brevig for providing helpful comments.
	\section{Background material}
	\subsection{The infinite polytorus and vertical limits}
	The infinite polytorus is defined as the (countable) infinite Cartesian product of copies of the unit  circle $\mathbb{T}$,
	$$\mathbb{T}^\infty=\left\{\chi=(\chi_1,\chi_2,\dots): \,\chi_j\in\mathbb{T},\, j\geq 1\right\}.$$
	It is a compact abelian group with respect to coordinate-wise multiplication. We can identify the Haar measure $m_\infty$ of the infinite polytorus with the countable infinite product measure $m\times m\times\cdots$, where $m$ is the normalized Lebesgue measure of the unit circle.
	
	By the prime number theorem, $\mathbb{T}^\infty$ is isomorphic to the group of characters of $(\mathbb{Q}_+,\cdot)$. Given a point $\chi=(\chi_1,\chi_2,\dots)\in\mathbb{T}^\infty$, the coresponding character $\chi:\mathbb{Q}_+\rightarrow\mathbb{T}$ is the completely multiplicative function on $\mathbb{N}$ such that $\chi(p_j)=\chi_j$, where $\{p_j\}_{j\geq1}$ is the increasing sequence of primes, extended to $\mathbb{Q}_+$ through the relation $\chi(n^{-1})=\overline{\chi(n)}$.
	
	Suppose $f(s)=\sum\limits_{n\geq1}\frac{a_n}{n^s}$ is a Dirichlet series and $\chi(n)$ is a character. The vertical limit function $f_\chi$ is defined as
	$$f_\chi(s)=\sum\limits_{n\geq1}\frac{a_n\chi(n)}{n^s}.$$
	The name comes from Kronecker's theorem \cite{BOH34}; for any $\epsilon>0$, there exists a sequence of real numbers $\{t_j\}_{j\geq1}$ such that $f(s+t_j)\rightarrow f_\chi(s)$ uniformly on $\mathbb{C}_{\sigma_u(f)+\epsilon}$.
	
	If $f\in \mathcal{D}_a$, then the abscissa of convergence satisfies $\sigma_c(f_\chi)\leq 0$ for almost every $\chi\in \mathbb{T}^\infty$. This is a consequence of the Rademacher--Menchov theorem \cite[Ch. XIII]{ZYG02}, following an argument of \cite{BAY02}. Finally, we note that if $\psi(s)=c_0s+\varphi(s)\in\mathfrak{G}$, and we set
	$$\psi_\chi(s)=c_0s+\varphi_\chi(s),$$
	then for every $\chi\in\mathbb{T}^\infty$ we have that
	\begin{equation} \label{eq:comprule}
	\left(C_\psi(f)\right)_\chi=f_{\chi^{c_0}}\circ\psi_\chi.
	\end{equation}
	\subsection{The hyperbolic metric and distance}
	The classical Schwarz--Pick lemma states that for every holomorphic self-map of the unit disk $\phi:\mathbb{D}\rightarrow\mathbb{D}$ and for any $z\in \mathbb{D}$,
	\begin{equation}\label{90}
	\frac{|\phi'(z)|}{1-|\phi(z)|^2}\leq \frac{1}{1-|z|^2}.
	\end{equation}
	Equality holds in \eqref{90} for one point $z_0\in\mathbb{D}$, and consequently for all points, if and only if $\phi$ is a holomorphic automorphism of the unit disk. The hyperbolic metric and distance in the unit disk are defined respectively as
	$$\lambda_{\mathbb{D}}(z)=\frac{2}{1-|z|^2}$$
	and
	$$d_{\mathbb{D}}(z,w)=\inf\limits_\gamma\int\limits_\gamma\lambda_{\mathbb{D}}(\zeta) \,|d\zeta|,$$
	where the infimum is taken over all piecewise smooth curves $\gamma$ in $\mathbb{D}$ that join $z$ and $w$. The 
	Schwarz--Pick lemma implies that every holomorphic self-map of the unit disk is a contraction of the hyperbolic distance,
	\begin{equation}\label{9}
	\lambda_{\mathbb{D}}(\phi(z))|\phi'(z)|\leq \lambda_{\mathbb{D}}(z),
	\end{equation}
	and
	\begin{equation}\label{10}
	d_\mathbb{D}(\phi(z),\phi(w)))\leq d_\mathbb{D}(z,w),
	\end{equation}
	where $z,\, w\in \mathbb{D}.$
	
	If equality holds in \eqref{9} for one point, or in \eqref{10} for a pair of distinct points, then $\phi$ is a holomorphic automorphism of the unit disk, and thus an isometry. Using the conformal invariance of the hyperbolic distance, one can prove that
	$$d_\mathbb{D}(z,w)=\log\frac{1+\left|\frac{w-z}{1-\overline{w}z}\right|}{1-\left|\frac{w-z}{1-\overline{w}z}\right|}=2\arctanh\left|\frac{w-z}{1-\overline{w}z}\right|.$$
	
	The Riemann mapping theorem allows us to transfer these notions to any simply connected proper subdomain $\Omega$ of the complex plane. More precisely, let $f$ be a Riemann map from $\Omega$ onto the unit disk. Then
	$$\lambda_{\Omega}(z)=\lambda_{\mathbb{D}}(f(z))|f'(z)|,$$
	and
	$$d_{\Omega}(z,w)=d_{\mathbb{D}}(f(z), f(w))=\inf\limits_\gamma\int\limits_\gamma\lambda_{\Omega}(\zeta) \, |d\zeta|,$$
	where the infimum is taken over all piecewise smooth curves $\gamma$ in $\Omega$ that join $z$ and $w$. By the Schwarz lemma it is easy to prove that $\lambda_{\Omega}$ and $d_{\Omega}$ are independent of the choice of the Riemann map. In the case of the right-half plane, considering the Riemann map $f(z)=\frac{z-w}{z+\overline{w}}$ we obtain that
	$$\lambda_{\mathbb{C}_0}(z)=\frac{1}{\Re z},$$
	and
	$$d_{\mathbb{C}_0}(z,w)=\log\frac{1+\left|\frac{z-w}{z+\overline{w}}\right|}{1-\left|\frac{z-w}{z+\overline{w}}\right|}=\log\frac{\left(|z+\overline{w}|+|z-w|\right)^2}{4\Re z\Re w},$$
	where $z,\,w\in \mathbb{C}_0$.
	
	The following Schwarz--Pick lemma for simply connected domains is a direct consequence of the definition and the ordinary Schwarz--Pick lemma. 
	\begin{theorem}[\cite{BM07}]\label{2.1}
		Suppose that $\Omega_1$ and $\Omega_2$ are simply connected proper subdomains of the complex plane and that $f:\Omega_1\rightarrow\Omega_2$ is a holomorphic function. Then, for every $z,$ $w\in\Omega_1$,
		\begin{equation}\label{11}
		\lambda_{\Omega_2}(f(z))|f'(z)|\leq \lambda_{\Omega_1}(z),
		\end{equation}
		and
		\begin{equation}\label{12}
		d_{\Omega_2}(f(z), f(w))\leq d_{\Omega_1}(z,w).
		\end{equation}
		Furthermore, equality holds  in \eqref{11} for one point, or in \eqref{12} for a pair of distinct points, if and only if $f$ is a biconformal map from $\Omega_1$ onto $\Omega_2$.
	\end{theorem}
	\section{Bounded composition operators on Bergman spaces of Dirichlet series}
	Consider the maps $T_\beta(z)=\beta\frac{1-z}{1+z},$ $\beta>0$, and $S_\theta(z)=z+\theta$, $\theta>0$, taking the unit disk $\mathbb{D}$ onto $\mathbb{C}_0$ and the half-plane $\mathbb{C}_0$ onto $\mathbb{C}_\theta$, respectively. Following \cite{GH99}, the space $H_i^2(\mathbb{C}_\theta,\beta)$ consists of those holomorphic functions on $\mathbb{C}_\theta$ such that $f\circ S_\theta\circ T_\beta\in H^2(\mathbb{\mathbb{D}})$, with norm
	\begin{equation*}
	\norm{f}^2_{H_i^2(\mathbb{C}_\theta,\beta)}:=\norm{f\circ S_\theta\circ T_\beta}_{H^2(\mathbb{D})}^2=\frac{\beta}{\pi}\int\limits_{-\infty}^{+\infty}|f(\theta+it)|^2\frac{dt}{\beta^2+t^2}.
	\end{equation*}
	We recall the following two lemmas.
	\begin{lemma}[\cite{GH99,QEF15}]\label{3.1}
		Let $f\in \mathcal{H}^2$ be such that $\sigma_u(f)\leq 0$. Then
		\begin{equation*}
		\lim_{\beta\rightarrow \infty}\norm{f}_{H_i^2(\mathbb{C}_0,\beta)}=\norm{f}_0.
		\end{equation*}
	\end{lemma}
	\begin{lemma}[\cite{BR17}]\label{3.2}
		For $\beta>0$ and $f\in \mathcal{H}^2$, 
		\begin{equation*}
		\norm{f}^2_{H_i^2(\mathbb{C}_\frac{1}{2},\beta)}\leq \max\left\{\frac{2}{\beta},\zeta(1+\beta)\right\}\norm{f}_0^2.
		\end{equation*}
	\end{lemma}
	The Cauchy--Schwarz inequality shows that point evaluations in $\mathbb{C}_{\frac{1}{2}}$ are bounded on $\mathcal{D}_a$. We record the following statement for easy reference.
	\begin{lemma}\label{3.3}
		Let $a\leq1$ and $\delta>0$. Then there exists a constant $C=C(a,\delta)$ such that for every $s\in\mathbb{C}_{\frac{1}{2}+\delta}$ and $f\in\mathcal{D}_a$,
		\begin{equation*}
		|f'(s)| \leq C\norm{f}_a|2^{-s}|.
		\end{equation*}
	\end{lemma}
	
	Our next goal is to establish a kind of Schwarz lemma for Dirichlet series. Note that the Schwarz--Pick lemma for the hyperbolic distance implies that 
	$$\liminf\limits_{|z|\rightarrow1^-}\frac{1-|\phi(z)|}{1-|z|}>0$$
	for any holomorphic self-map $\phi$ of $\mathbb{D}$, see  \cite[Lemma 1.4.5]{BCDM20}. The corresponding inequality does not hold for all self-maps of the right half-plane. However, for a Dirichlet series $\varphi\in\mathfrak{G}_0$ we will prove that
	\begin{equation}\label{15}
	\liminf\limits_{\Re s\rightarrow0^+}\frac{\Re\varphi(s)-\frac{1}{2}}{\Re s}=\delta>0.
	\end{equation} 
	This implies a quantitative version of \cite[Prop. 4.2]{GH99}. Namely, that for sufficiently small $\epsilon>0$,
	\begin{equation*}
	\varphi(\mathbb{C}_\epsilon)\subset\mathbb{C}_{\frac{1}{2}+\epsilon\delta}.
	\end{equation*}
	The key idea in proving \eqref{15} is to exploit the vertical translations of $\varphi \in\mathfrak{G}_0$ to restrict the limit to a half-strip, where the quantity in \eqref{15} can be shown to be uniformly bounded from below by virtue of Theorem~\ref{2.1}.
	\begin{lemma}\label{3.4}
		For every $\varphi\in \mathfrak{G}_0$ there exists a constant $C=C(\varphi)>0$ such that
		\begin{equation*}
		\Re s\leq C \left(\left(\Re s\right)^2+1\right)\left(\Re \varphi(s)-\frac{1}{2}\right),\qquad s\in\mathbb{C}_0.
		\end{equation*}
	\end{lemma}
	\begin{proof}
We consider the vertical translations $\varphi(s+it)=\varphi_{\chi_t}(s)$, where $\chi_t=n^{-it},\,t\in\mathbb{R}$. Observing that $\varphi_{\chi_t}\in\mathfrak{G}_0$, we have
		\begin{align*}
		\log\left(\frac{\lambda_{\mathbb{C}_0}(z)}{\lambda_{\mathbb{C}_0}(\varphi_{\chi_t}(z)-\frac{1}{2})}\right)&=\log\left(\frac{\Re \varphi_{\chi_t}(z)-\frac{1}{2}}{\Re z}\right)\\
		&=\log\left(\frac{\left(|z+1|+|z-1|\right)^2}{4\Re z}\right)-\log\left(\frac{\left(|\varphi_{\chi_t}(z)+\frac{1}{2}|+|\varphi_{\chi_t}(z)-\frac{3}{2}|\right)^2}{4(\Re \varphi_{\chi_t}(z)-\frac{1}{2})}\right)\\
		&+2\log\left(\frac{|\varphi_{\chi_t}(z)+\frac{1}{2}|+|\varphi_{\chi_t}(z)-\frac{3}{2}|}{|z+1|+|z-1|}\right)\\
		&\geq d_{\mathbb{C}_0}(z,1)-d_{\mathbb{C}_0}(\varphi_{\chi_t}(z)-\frac{1}{2},1)+2\log\left(\frac{2}{|z+1|+|z-1|}\right).
		\end{align*}
		By the Schwarz--Pick lemma and the triangle inequality for the hyperbolic distance, we find from here that
		\begin{equation} \label{18}
		\log\left(\frac{\Re\varphi_{\chi_t}(z)-\frac{1}{2}}{\Re z}\right) \geq -d_{\mathbb{C}_0}(\varphi_{\chi_t}(1)-\frac{1}{2}, 1)+2\log\left(\frac{2}{|z+1|+|z-1|}\right).
		\end{equation}
		The crucial step is to note that the quantity $d_{\mathbb{C}_0}(\varphi_{\chi_t}(1)-\frac{1}{2}, 1)$ is uniformly bounded, since $\varphi$ maps the line $\Re z = 1$ into a compact subset of $\mathbb{C}_{1/2}$. Given $s \in \mathbb{C}_0$, we can therefore choose $z = \Re s$ and $t = \Im s$ to obtain that
		 		\begin{equation*}
		 	\Re s\leq C \left[\left(|z+1|+|z-1|\right)^2\right]\left(\Re\varphi(s)-\frac{1}{2}\right),
		 \end{equation*}
	 which is the desired inequality.
	\end{proof}

We next recall Littlewood's subordination principle, which implies that any holomorphic self-map of the unit disk generates a bounded composition operator on the Hardy space $H^2(\mathbb{D})$.
	\begin{lemma}[\cite{LIT25,SHAP93}]\label{3.5}
		Suppose $\phi$ is a holomorphic self-map of the unit disk $\mathbb{D}$. Then, for every $f\in H^2(\mathbb{D})$,
		\begin{equation*}
		\norm{f\circ\phi}_{H^2(\mathbb{D})}\leq \sqrt{\frac{1+|\phi(0)|}{1-|\phi(0)|}}\norm{f}_{H^2(\mathbb{D})}.
		\end{equation*}
	\end{lemma}
	We also borrow the following lemma from \cite{GH99}.
	\begin{lemma}[\cite{GH99}]\label{3.6}
		Let $a\leq1$ and let $\{p_j\}_{j\geq1}$ be the increasing sequence of primes. Then, the function $f(s)=\sum\limits_{j\geq1}a_{p_j}p_j^{-s}$, with coefficients  $a_{p_j}=\left(\sqrt{p_j}\log(p_j)^{1+\frac{a}{2}}\right)^{-1}$, satisfies the following:
	\end{lemma}
	\begin{enumerate}[(i)]
		\item  $f\in\mathcal{D}_a$ and $\sigma_c(f)=\frac{1}{2}$.
		\item $\sigma_c(f_\chi)=0$, for almost every $\chi\in\mathbb{T}^\infty$.\\
	\end{enumerate}
	
	As promised in the introduction, we now provide a proof of the characterization of the bounded composition operators on the Bergman spaces $\mathcal{D}_a,$ $a\leq0$, which is new for Dirichlet series symbols. To do so, we will combine the original argument of Gordon and Hedenmalm \cite{GH99} with the Schwarz lemma for Dirichlet series.
	\begin{theorem}[\cite{Bail15, BB16}]\label{3.7}
		For $a > 0$, the class $\mathfrak{G}$ determines all bounded composition operators on the Bergman space of Dirichlet series $\mathcal{D}_{-a}.$
	\end{theorem}
	\begin{proof}
		It was essentially already proven in \cite{GH99} that it is necessary that $\psi \in \mathfrak{G}$ in order for $C_\psi \colon \mathcal{D}_a \to \mathcal{D}_a$ to be bounded. Indeed, by \cite[Theorem  8.3.1]{QQ20}, $P\circ \psi$ is a Dirichlet series for every polynomial $P$ if and only if the symbol $\psi:\mathbb{C}_\frac{1}{2}\rightarrow \mathbb{C}_\frac{1}{2}$ has the form $\psi(s)=c_0s+\varphi(s)$, where $c_0$ is a non-negative integer and $\varphi$ is a Dirichlet series. The mapping properties of $\psi$ are deduced from the composition rule \eqref{eq:comprule}  and Lemma~\ref{3.6}, noting that  $\varphi(\mathbb{C}_0)=\varphi_\chi(\mathbb{C}_0)$ for every $\chi\in\mathbb{T}^\infty.$
		
		Conversely, suppose $\psi \in \mathfrak{G}$. Let $\mu$ denote the probability measure  $d\mu(\sigma) = \frac{2^a}{\Gamma(a)}\sigma^{a-1}e^{-2\sigma} \, d\sigma$ on $(0,\infty)$, observing that
		\begin{equation*}
		\norm{f}_{-a}^2\approx\int\limits_{\mathbb{T}^\infty}\int\limits_0^{\infty}\int\limits_{0}^1|f_\chi(\sigma+it)|^2 \,dt \, d\mu(\sigma) \, dm_\infty(\chi)=\int\limits_0^{\infty}\norm{f_\sigma}^2_0 \, d\mu(\sigma),
		\end{equation*}
		where $f_\sigma=f(\cdot+\sigma)$.
		First we will consider the case when $\psi(s)=c_0s+\varphi(s)\in \mathfrak{G}_{\geq 1}$. In this case the analogue of the Schwarz lemma is trivial: $\Re s\leq \Re\psi(s)$. For a Dirichlet polynomial $f$ and a positive number $\beta>0$, we define the functions
		$$F_\beta=f\circ S_{\sigma}\circ T_\eta$$
		and 
		$$g_\beta=T_\eta^{-1}\circ S^{-1}_{\sigma}\circ\psi_\sigma\circ T_\beta,$$
		where $\eta=c_0(\beta+\sigma)-\sigma$. Note that $g_\beta(0) \to 0$ as $\beta \to \infty$. By Lemma \ref{3.1} and Lemma \ref{3.5}, we have
		\begin{align*}
		\norm{C_{\psi_\sigma}(f)}_0&=\lim\limits_{\beta\rightarrow\infty}\norm{f\circ\psi_\sigma}_{H_i^2(\mathbb{C}_0,\beta)}
		=\lim\limits_{\beta\rightarrow\infty}\norm{F_\beta\circ g_\beta}_{H^2(\mathbb{D})}
		\leq \lim\limits_{\beta\rightarrow\infty}\sqrt{\frac{1+|g_\beta(0)|}{1-|g_\beta(0)|}}\norm{F_\beta}_{H^2(\mathbb{D})}\\
		&= \lim\limits_{\eta\rightarrow\infty}\norm{f\circ S_{\sigma}\circ T_\eta}_{H^2(\mathbb{D})} = \norm{f_\sigma}_0.
		\end{align*}
		Therefore
		\begin{align*}
		\norm{C_\psi(f)}_{-a}^2&\approx\int\limits_0^{\infty}\norm{C_{\psi_\sigma}(f)}^2_0 \, d\mu(\sigma)\leq\int\limits_0^{\infty}\norm{f_\sigma}^2_0 \, d\mu(\sigma)\approx\norm{f}_{-a}^2,
		\end{align*}
		which demonstrates that the composition operator is bounded in this case.
		
		Suppose next that $\varphi\in \mathfrak{G}_0$. By a vertical translation of the argument $f$, there is no loss of generality in assuming that $\varphi(+\infty) > 1/2$. By Lemma \ref{3.4} there exists a constant $\lambda=\lambda(\varphi)>0$ such that 
		\begin{equation*}
		\lambda \Re s\leq \Re\varphi(s)-\frac{1}{2}, \qquad 0 < \Re s < 1.
		\end{equation*}
		In this case, for a Dirichlet polynomial $f$ and positive numbers $\beta>0$, we define the functions
		$$F=f\circ S_{\lambda\sigma+\frac{1}{2}}\circ T_\eta$$
		and
		$$g_\beta=T_\eta^{-1}\circ S^{-1}_{\lambda\sigma+\frac{1}{2}}\circ\varphi_\sigma\circ T_\beta,$$
		where $\eta= \varphi(+\infty)-\lambda\sigma-\frac{1}{2}$ and $0<\sigma<\delta := \frac{\varphi(+\infty)-\frac{1}{2}}{2\lambda}$. Then we again have that $\lim_{\beta \to \infty} g_\beta(0) = 0$, and
		\begin{align*}
		\norm{C_{\varphi_\sigma}(f)}_0&=\lim\limits_{\beta\rightarrow\infty}\norm{f\circ\varphi_\sigma}_{H_i^2(\mathbb{C}_0,\beta)}
		=\lim\limits_{\beta\rightarrow\infty}\norm{F\circ g_\beta}_{H^2(\mathbb{D})} \\
		&\leq \lim\limits_{\beta\rightarrow\infty}\sqrt{\frac{1+|g_\beta(0)|}{1-|g_\beta(0)|}}\norm{F}_{H^2(\mathbb{D})} = \norm{f_{\lambda\sigma}}^2_{H_i^2(\mathbb{C}_\frac{1}{2},\eta)}.
		\end{align*}
		By Lemma \ref{3.2} we conclude that there is a constant such that
		$$\norm{C_{\varphi_\sigma}(f)}_0\leq C \norm{f_{\lambda\sigma}}_0, \qquad 0 < \sigma < \delta.$$
		For $\sigma \geq \delta$, we simply note that $\varphi_\sigma(\mathbb{C}_0) \subset \mathbb{C}_{1/2 + \varepsilon}$ for some $\varepsilon > 0$, and therefore by the Cauchy--Schwarz inequality that 
		$$\sup_{s \in \mathbb{C}_0} |f(\varphi_\sigma(s))| \leq C \|f\|_{-a}.$$
		Hence $\|C_{\varphi_\sigma}(f)\|_0 \leq C \|f\|_{-a},$ as can be seen for example from Carlson's theorem, see \cite[Lemma~3.2]{HLS97}.
		We conclude that
		\begin{equation*}
		\norm{C_\varphi(f)}_{-a}^2\approx\int\limits_0^{\infty}\norm{C_{\varphi_\sigma}(f)}^2_0\, d\mu(\sigma)\\
		\leq C \int\limits_0^{\delta}\norm{f_{\lambda\sigma}}^2_0d\mu(\sigma)+C \norm{f}_{-a}^2
		\leq C\norm{f}^2_{-a}. \qedhere
		\end{equation*}
	\end{proof}
	\begin{remark}
		Using the same argument one can prove that Theorem~\ref{3.7} holds for all Bergman-like spaces of Dirichlet series \cite{MCC04},
		\begin{equation*}
		\mathcal{D}_\mu=\left\{f(s)=\sum_{n\geq1}\frac{a_n}{n^s}:\norm{f}_\mu^2=\sum_{n\geq1}|a_n|^2w_\mu(n)< \infty \right\},
		\end{equation*}
		assuming that the coefficients are of the form $$w_\mu(n)=\int\limits_{0}^\infty\frac{d\mu(\sigma)}{n^{2\sigma}},$$
		where $\mu$ is a probability measure on $(0,\infty)$ with $0\in \supp(\mu)$ and satisfying 
		\begin{equation}\label{21}
		\int\limits_{0}^\infty\frac{d\mu(\sigma)}{n^{2\lambda\sigma}}\leq C(\lambda)\int\limits_{0}^\infty\frac{d\mu(\sigma)}{n^{2\sigma}}, \qquad 0 < \lambda < 1.
		\end{equation}
		Every symbol $\psi\in \mathfrak{G}_{\geq1}$ induces a contraction $C_\psi$ on $\mathcal{D}_{\mu}$, even without the condition \eqref{21}.
	\end{remark}
	\section{Weighted mean counting functions}
	In this section, we will investigate the properties of the weighted counting function $M_{\varphi, a}(w,\sigma),$ $\sigma>0$, where $\varphi$ is a Dirichlet series with abscissa of convergence $\sigma_u(\varphi)\leq 0$. Firstly, we will prove the existence of this function, generalizing \cite[Theorem~6.2]{BP21}. Monotonicity then ensures the existence of the limit function $M_{\varphi, a}(w)$ (finitely or infinitely). Secondly, following the ideas of Aleman \cite{AL92}  and Shapiro \cite{SHAP87} from the disk case, we will give a weak version of the submean value property for the weighted counting function $M_{\varphi,a}(w)$, $a > 0$. Note that the (strong) submean value property of $M_{\varphi, 1}(w)$ was proven in \cite[Lemma~6.5]{BP21}.
	\subsection{Existence}
	In \cite{BP21}, the existence of $M_{\varphi, 1}(w,\sigma)$ was established through Littlewood's lemma \cite[Sec. 9.9]{TIT86}, which is a rectangular version of Jensen's formula \cite[Sec. 10.2]{SHAP87}. We will replace Littlewood's lemma with the following theorem, which allows us to count the zeros of a non-zero holomorphic function in an arbitrary domain.
	\begin{theorem}[\cite{RAN95}]\label{4.1}
		Let $u\not\equiv-\infty$ be a subharmonic function on a domain $\Omega$ in $\mathbb{C}$. Then, there exists a unique Radon measure $\Delta u$ on $\Omega$ such that for every compactly supported function $v \in C^\infty(\Omega)$, it holds that 
		\begin{equation*}
		\int\limits_{\Omega}v\Delta u=\int\limits_{\Omega}u\Delta v \, dA.
		\end{equation*}
		In the special case that $u=\log|f|$, where  $f\not\equiv0$ is a holomorphic function on the domain $\Omega$, the measure $\frac{1}{2\pi}\Delta u$ is the sum of Dirac masses at the zeros of $f$, counting multiplicity.
	\end{theorem}
	The almost periodicity of the Dirichlet series $\varphi$ in $\mathbb{C}_{\sigma_0}$, $\sigma_0>0$, implies an argument principle for the unweighted counting function $M_{\varphi, 0}$, see \cite{JT45}.
	\begin{lemma}\label{4.2}
		Suppose that $\varphi$ is a Dirichlet series with abscissa of uniform convergence $\sigma_u(\varphi)\leq 0$. If $\varphi(+\infty)\neq0$, $\{\Re s=\sigma_0\}$ is a zero-free line for the function $\varphi$ and $\{T_j\}_{j\geq1}$ is an increasing sequence  of positive real numbers, relatively dense in $[0,+\infty)$, such that 
		$$|\varphi(\sigma+iT_j)|\geq \delta>0, \qquad \sigma\geq\sigma_0,$$
		then
		$$M_{\varphi, 0}(0,\sigma_0)=-\lim_{j\rightarrow\infty}\frac{1}{2 T_j}\int\limits_{-T_j}^{T_j}\frac{\varphi'(\sigma_0+it)}{\varphi(\sigma_0+it)}dt.$$
	\end{lemma}
	\begin{proof}
		Let $\sigma_\infty>0$ be such that the equation $\varphi(s)=0$ has no solution for $\Re s\geq\sigma_\infty$.
		We will denote by $R_j$ the rectangle with vertices at $\sigma_0\pm iT_j,\,\sigma_\infty\pm iT_j$. By the argument principle, we then have that
		\begin{multline*}
		M_{\varphi,0}(0, \sigma_0, T_j)=\frac{1}{2 iT_j}\int\limits_{\partial R_j}\frac{\varphi'(\zeta)}{\varphi(\zeta)}d\zeta = \\
		\frac{1}{2T_j}\left(-\int\limits_{-T_j}^{T_j}\frac{\varphi'(\sigma_0+it)}{\varphi(\sigma_0+it)}dt + i\int\limits_{\sigma_0}^{\sigma_\infty}\frac{\varphi'(\sigma+iT_j)}{\varphi(\sigma+iT_j)}d\sigma - i\int\limits_{\sigma_0}^{\sigma_\infty}\frac{\varphi'(\sigma-iT_j)}{\varphi(\sigma-iT_j)}d\sigma+\int\limits_{-T_j}^{T_j}\frac{\varphi'(\sigma_\infty+it)}{\varphi(\sigma_\infty+it)}dt\right).
		\end{multline*}
		We observe that the first coefficient of the Dirichlet series $f=\frac{\varphi'}{\varphi}$ satisfies  $f(+\infty)=0$. Thus, letting $T_j\rightarrow \infty$ and then $\sigma_\infty\rightarrow\infty$ follows that
		\begin{equation*}
		M_{\varphi, 0}(0,\sigma_0)=\lim_{j\rightarrow\infty}M_{\varphi,0}(0, \sigma, T_j)= -\lim_{j\rightarrow\infty}\frac{1}{2 T_j}\int\limits_{-T_j}^{T_j}\frac{\varphi'(\sigma_0+it)}{\varphi(\sigma_0+it)}dt. \qedhere
		\end{equation*}
	\end{proof}
	We begin by proving a special case of Theorem \ref{1.1}.
	\begin{theorem}\label{4.3}
		Let $\varphi$ be a Dirichlet series such that $\sigma_u(\varphi)\leq 0$,  and let $w \neq \varphi(+\infty)$ be such that $\{\Re s=\sigma_0\}$ is a zero free line for the function $\varphi-w$. Then, for every  $a\in\mathbb{R}$, the counting function $M_{\varphi,a}(w,\sigma_0)$ exists and satisfies
		\begin{equation}\label{23}
		M_{\varphi,a}(w,\sigma_0)=M_{\varphi,0}(w,\sigma_0)\sigma_0^{a}+a\int\limits_{\sigma_0}^{\infty}t^{a-1}M_{\varphi,0}(w,t) \, dt.
		\end{equation}
		Furthermore, for sufficiently large $\sigma_\infty  > 0$,
		\begin{equation}\label{24}
		\begin{aligned}
		M_{\varphi,a}(w,\sigma_0) &- M_{\varphi,0}(w,\sigma_0)\sigma_0^{a} = \\ &a\sigma_0^{a-1}\mathcal{J}_{\varphi-w}(\sigma_0)-a\sigma_\infty^{a-1}\log|\varphi(+\infty)-w|-a(1-a)\int\limits_{\sigma_0}^{\sigma_\infty}t^{a-2}\mathcal{J}_{\varphi-w}(t)dt,
		\end{aligned}
		\end{equation}
		where $\mathcal{J}_{\varphi - w}$ is the Jessen function \eqref{eq:jessenfcn}.
	\end{theorem}
	\begin{proof}
		Without loss of generality we assume that $w=0$.  By almost periodicity there exists an increasing sequence $\{T_j\}_{j\geq1}$ of positive real numbers, relatively dense in $[0,+\infty)$, such that for every $\sigma\geq\sigma_0$, $$|\varphi(\sigma \pm iT_j)|\geq \delta>0.$$
		Let $\sigma_\infty>0$ be so large that that $\varphi\neq0$ in $\mathbb{C}_{\frac{\sigma_\infty}{2}}$.
		Then $\Delta\log|\varphi|=0$ near the boundary of the rectangle  $R_j$ with vertices at $\sigma_0\pm i T_j,$  $\sigma_\infty\pm i T_j$.
		
		By a  $C^\infty$ version of Urysohn's lemma \cite[Theorem 8.18]{FOL99} there exists a function $\psi\in C_c^\infty(R_j)$ such that $\psi(s) = \left(\Re s\right)^a$ for $s \in \supp(\Delta\log|\varphi|)$. Theorem \ref{4.1} implies that
		\begin{align*}
		\int\limits_{R_j}\left(\Re z\right)^{a}\Delta\log|\varphi(z)| &=\int\limits_{R_j}\psi(z)\Delta\log|\varphi(z)| \\
		&= 2\pi\sum_{\substack{s\in\varphi^{-1}(\{0\})\\s\in R_j}}\left(\Re s\right)^a = 2T_jM_{\varphi, a}(0,\sigma_0,T_j).
		\end{align*}
		On the other hand, Green's theorem implies that
		\begin{align*}
		&\int\limits_{R_j}\left(\Re z\right)^{a}\Delta\log|\varphi(z)| + a(1-a)\int\limits_{R_j}\left(\Re z\right)^{a-2}\log|\varphi(z)|dA(z)\\
		&=\ointctrclockwise\limits_{\partial R_j}\left(\Re\zeta\right)^{a}\left(-\frac{d\log|\varphi(\zeta)|}{dy},\frac{d\log|\varphi(\zeta)|}{dx}\right)\cdot d\zeta-\ointctrclockwise\limits_{\partial R_j}\log|\varphi(\zeta)|\left(-\frac{d\left(\Re\zeta\right)^{a}}{dy},\frac{\left(\Re\zeta\right)^{a}}{dx}\right)\cdot d\zeta,
		\end{align*}
		where $\zeta=x+iy$. For the first line integral on the right-hand side, we have that
		\begin{multline*}
		\ointctrclockwise\limits_{\partial R_j}\left(\Re\zeta\right)^{a}\left(-\frac{d\log|\phi(\zeta)|}{dy},\frac{d\log|\phi(\zeta)|}{dx}\right)\cdot d\zeta = \\ -\sigma_0^{a}\Re\left(\int\limits_{-T_j}^{T_j}\frac{\phi'(\sigma_0+ it)}{\phi(\sigma_0+it)}dt\right)+\sigma_\infty^{a}\Re\left(\int\limits_{-T_j}^{T_j}\frac{\phi'(\sigma_\infty+ it)}{\phi(\sigma_\infty+it)}dt\right) \pm \Re\left(\int\limits_{\sigma_0}^{\sigma_\infty}\sigma^{a}\frac{i\phi'(\sigma\pm iT_j)}{\phi(\sigma \pm iT_j)}d\sigma\right).
		\end{multline*}
		From Lemma \ref{4.2}, dividing through by $2T_j$ and letting $j\rightarrow\infty$, we obtain that
		\begin{equation*}
		\lim_{j\rightarrow\infty}\frac{1}{2T_j}\ointctrclockwise\limits_{\partial R_j}\left(\Re\zeta\right)^{a}\left(-\frac{d\log|\phi(\zeta)|}{dy},\frac{d\log|\phi(\zeta)|}{dx}\right)\cdot d\zeta= M_{\phi,0}(0,\sigma_0)\sigma_0^{a}.
		\end{equation*}
		Writing out the second line integral,
		\begin{multline*}
		\ointctrclockwise\limits_{\partial R_j}\log|\phi(\zeta)|\left(-\frac{d\left(\Re\zeta\right)^{a}}{dy},\frac{\left(\Re\zeta\right)^{a}}{dx}\right)\cdot d\zeta = \\-a\sigma_0^{a-1}\int\limits_{-T_j}^{T_j}\log|\phi(\sigma_0+it)|dt
		+a\sigma_\infty^{a-1}\int\limits_{-T_j}^{T_j}\log|\phi(\sigma_\infty+it)|dt,
		\end{multline*}
		we have that 
		$$\lim_{j\rightarrow\infty}\frac{1}{2T_j}\ointctrclockwise\limits_{\partial R_j}\log|\phi(\zeta)|\left(-\frac{d\left(\Re\zeta\right)^{a}}{dy},\frac{\left(\Re\zeta\right)^{a}}{dx}\right)\cdot d\zeta=-a\sigma_0^{a-1}J_\phi(\sigma_0)+a\sigma_\infty^{a-1}J_\phi(\sigma_\infty),$$
		where $J_\phi(\sigma_\infty)=\log|\phi(+\infty)|$ by \cite[Theorem~31]{JT45}.
		
		We apply Fubini's theorem to the area integral,
		\begin{equation*}
		\frac{1}{2T_j}\int\limits_{R_j}\left(\Re z\right)^{a-2}\log|\phi(z)|dA(z)=\int\limits_{\sigma_0}^{\sigma_\infty}\sigma^{a-2}\frac{1}{2T_j}\int\limits_{-T_j}^{T_j}\log|\phi(\sigma+it)|dtd\sigma.
		\end{equation*}
		Since the sequence of functions $\left\{\frac{1}{T_j}\int\limits_{-T_j}^{T_j}\log|\varphi(\sigma+it)|dt\right\}_{j\geq1
		}$ is uniformly bounded on $[\sigma_0,+\infty)$ by \cite[Theorem 5]{JT45}, the dominated convergence theorem implies that
		\begin{align*}
		a(1-a)\lim_{j\rightarrow\infty}\frac{1}{2T_j}\int\limits_{R_j}\left(\Re z\right)^{a-2}\log|\varphi(z)|dA(z)=a(1-a)\int\limits_{\sigma_0}^{\sigma_\infty}\sigma^{a-2}\mathcal{J}_\varphi(\sigma)d\sigma.
		\end{align*}
		We conclude that
		\begin{multline*}
		\lim_{j\rightarrow\infty}M_{\varphi,a}(0,\sigma_0,T_j) - M_{\varphi,0}(0,\sigma_0)\sigma_0^{a} = \\ a\sigma_0^{a-1}\mathcal{J}_\varphi(\sigma_0)-a\sigma_\infty^{a-1}\log|\varphi(+\infty)| -a(1-a)\int\limits_{\sigma_0}^{\sigma_\infty}\sigma^{a-2}\mathcal{J}_\varphi(\sigma)d\sigma.
		\end{multline*}
		This finishes the proof of \eqref{24}, since almost periodicity and the argument principle show that the number of zeros of $\varphi$ on any rectangle $(\sigma_0,\sigma_\infty)\times(T-d,T+d)$ is uniformly bounded, where $d=\sup\limits_{j\geq 1}(T_{j+1}-T_j)$, see for example \cite[Theorem~3]{JT45}.
		
		Finally, the Jessen function $\mathcal{J}_\varphi(\sigma)$ is convex and, as a consequence, absolutely continuous on every closed sub-interval of the positive semi-axis. Thus, we can integrate by parts, yielding that
		\begin{equation*}
		M_{\varphi,a}(0,\sigma_0)=M_{\varphi,0}(0,\sigma_0)\sigma_0^{a}+a\int\limits_{\sigma_0}^{\infty}t^{a-1}M_{\varphi,0}(0,t)dt,
		\end{equation*}
		which is  \eqref{23}.
	\end{proof}
	Before proving Theorem \ref{1.1}, we extract the following technical lemma from the work of \cite[Lemma 2.4]{BP21}.
	\begin{lemma}\label{4.4}
		Let $\varphi$ be a Dirichlet series such that $\sigma_u(\varphi)\leq0$ and $\varphi(+\infty)\neq0$. Then for every $\sigma_0>0$, and for $T>0$ sufficiently large, there exists a constant $C(\sigma_0,\varphi)>0$ such that,
		\begin{equation}\label{28}
		\frac{\pi}{T}\sum\limits_{\substack{s\in\varphi^{-1}(\{0\})\\
				|\Im s|<T\\
				\sigma<\Re s<\infty}} 1 \leq C, \qquad \sigma\geq\sigma_0.
		\end{equation}
	\end{lemma}
	\begin{proof}
		Let $\Theta$ denote the unique conformal map from the unit disk to the half-strip
		$$S_1=\{s:\Re s>0,\, |\Im s|<1\}$$
		with $\Theta(0)=1$ and $\Theta'(0)>0$. We observe that
		\begin{equation*}
		\Theta^{-1}(s)=\frac{\sinh(\frac{s\pi}{2})-\sinh(\frac{\pi}{2})}{\sinh(\frac{s\pi}{2})+\sinh(\frac{\pi}{2})},
		\end{equation*}
		and that there exists absolute constants $\delta_1$, $\delta_2>0$ such that $$\delta_1<|\left(\Theta^{-1}\right)'(s)|<\delta_2$$
		whenever $|\Im s|\leq \frac{1}{2}$ and $0\leq \Re s\leq \frac{1}{2}$. The Koebe quarter theorem \cite[Corollary 1.4]{POM92} implies that for every $s\in S_1$,
		\begin{equation*}
		\frac{1-|\Theta^{-1}(s)|^2}{4\left|\left(\Theta^{-1}\right)'(s)\right|}\leq \dist(s,\partial S_1)\leq \frac{1-|\Theta^{-1}(s)|^2}{\left|\left(\Theta^{-1}\right)'(s)\right|}.
		\end{equation*}
		Thus, there exists an absolute constant $C_0$ such that
		\begin{equation}\label{30}
		\pi \Re s\leq C_0 \log\left|\frac{1}{\Theta^{-1}(s)}\right|,
		\end{equation}
		when $|\Im s|\leq \frac{1}{2}$ and $0\leq \Re s\leq \frac{1}{2}$.
		
		For $T>0$ we will denote by $S_T$ the half-strip $TS_1$ and by $\Theta_T:\mathbb{D}\rightarrow S_T$ the map $\Theta_T=T\Theta$. We consider the function $\varphi_{\frac{\sigma_0}{2}}(s)=\frac{\varphi(s+\frac{\sigma_0}{2})}{M}$, where $M=\sup\limits_{\mathbb{C}_{\frac{\sigma_0}{2}}}|\varphi(s)|$. Then, for $T$ so large that the equation $\varphi(s)=0$ has no solutions for $\Re s>T$, we have by \eqref{30} that
		\begin{align*}
		\frac{\pi}{T}\sum\limits_{\substack{s\in\varphi^{-1}(\{0\})\\
				|\Im s|<T\\
				\sigma<\Re s<\infty}}1&\leq C\frac{\pi}{T}\sum\limits_{\substack{s\in\varphi^{-1}(\{0\})\\
				|\Im s|<T\\
				\sigma<\Re s<\infty}}\left(\Re s-\frac{\sigma_0}{2}\right)\leq C\frac{\pi}{T}\sum\limits_{\substack{s\in\varphi_{\frac{\sigma_0}{2}}^{-1}(\{0\})\\
				|\Im s|<T\\
				\Re s>0}}\Re s\\&\leq C\sum\limits_{\substack{s\in\varphi_{\frac{\sigma_0}{2}}^{-1}(\{0\})\\
				|\Im s|<2T\\
				\Re s>0}}\log\left|\frac{1}{\Theta^{-1}_{2T}(s)}\right|=C N_{\psi_{2T}}(0),
		\end{align*}
		where $\psi_{2T}=\varphi_{\frac{\sigma_0}{2}}\circ\Theta_{2T}$ and $N_{\psi_{2T}}$ is the classical Nevanlinna counting function. Thus, the Littlewood inequality \cite{SHAP87},
		$$N_{\psi_{2T}}(0) \leq \log \left| \frac{1}{\psi_{2T}(0)}\right|,$$
		implies that
		\begin{equation*}
		\frac{\pi}{T}\sum\limits_{\substack{s\in\varphi^{-1}(\{0\})\\
				|\Im s|<T\\
				\sigma<\Re s<\infty}}1\leq C \log\left|\frac{M}{\varphi(\frac{\sigma_0}{2}+2T)}\right| \qedhere
		\end{equation*}
	\end{proof}	
	We can now give the proof of Theorem~\ref{1.1}.
	\begin{proof}[\textbf{Proof of Theorem \ref{1.1}}]
		Without loss of generality we can assume that $w=0$.
		Let $\omega(\sigma)=\sigma^a.$ Then,
		\begin{align*}
		\int\limits_{\sigma_0}^{\infty}\omega'(\sigma)M_{\varphi,0}(0,\sigma)d\sigma&=\int\limits_{\sigma_0}^{\sigma_\infty}\omega'(\sigma)\lim_{T\rightarrow\infty}\frac{\pi}{T}\sum\limits_{\substack{s\in\varphi^{-1}(\{0\})\\
				|\Im s|<T\\
				\sigma<\Re s<\infty}}1 \, d\sigma,
		\end{align*}
		where $\sigma_\infty>0$ is such that the equation $\varphi(s)=0$ has no solutions in $\mathbb{C}_{\sigma_\infty}$.
		By the dominated convergence theorem, which applies in light of \eqref{28}, and then Fubini's theorem, we obtain that
		\begin{align*}
		\int\limits_{\sigma_0}^{\infty}\omega'(\sigma)M_{\varphi,0}(0,\sigma) \, d\sigma&=\lim_{T\rightarrow\infty}\frac{\pi}{T}\int\limits_{\sigma_0}^{\infty}\omega'(\sigma)\sum\limits_{\substack{s\in\varphi^{-1}(\{0\})\\
				|\Im s|<T\\
				\sigma<\Re s<\infty}}1 \,d\sigma
		=\lim_{T\rightarrow\infty}\frac{\pi}{T}\sum\limits_{\substack{s\in\varphi^{-1}(\{0\})\\
				|\Im s|<T\\
				\sigma_0<\Re s<\infty}}\int\limits_{\sigma_0}^{\Re s}\omega'(\sigma) \,d\sigma\\
		&=M_{\varphi,a}(0,\sigma_0)-\sigma_0^aM_{\varphi,0}(0,\sigma_0).
		\end{align*}
		This proves the existence of the function $M_{\varphi,a}(0,\sigma_0)$ and \eqref{3}. The right continuity of $M_{\varphi, a}(0, \sigma)$ is now a consequence of the right continuity of $M_{\varphi, 0}(0, \sigma)$, see \cite[Lemma~5.1]{BP21}. Integrating by parts as in the proof of Theorem \ref{4.3}, we also obtain \eqref{eq:meanintbyparts}.
	\end{proof}
	Strictly speaking, this argument is independent of Theorem \ref{4.3}. However, we find the proof of Theorem~\ref{4.3} to be illuminating and interesting in its own right. Note that the proof for Theorem~\ref{1.1} can also be applied to the more general counting function induced by a twice continuously differentiable weight $\omega(s)=\omega(\Re s)$ on $(0, \infty)$,
	$$M_{\varphi,\omega}(w, \sigma, T)=\frac{\pi}{T}\sum\limits_{\substack{s\in\varphi^{-1}(\{w\})\\
			|\Im s|<T\\
			\sigma<\Re s<\infty}}\omega(s), \qquad w\neq\varphi(+\infty).$$
	
	By monotonicity we deduce the following.
	\begin{corollary}
		Let $\varphi$ be a Dirichlet series such that $\sigma_u(\varphi)\leq 0$ and $\varphi(+\infty)\neq w$. Then, for every  $a\in\mathbb{R}$ the counting function $M_{\varphi,a}(w)=\lim\limits_{\sigma\rightarrow0^+}M_{\varphi,a}(w,\sigma)$ exists, finitely or infinitely.
	\end{corollary}
	The limit is not finite in general for $a \geq 0$, as we now exemplify.
	\begin{example}\label{4.7}
		Applying the transference principle \cite{QS15} to the example constructed by Zorboska in \cite{ZOR98}, we obtain a Dirichlet series $\varphi$ such that the $a$-weighted counting function is finite if and only if $a>\frac{1}{2}$. More precisely, we consider the Dirichlet series $\varphi(s)=g(2^{-s}),$ where $g(z)=e^{-\frac{1+z}{1-z}},\,z\in\mathbb{D}$. We observe that $\varphi$ is a periodic function (with period $ip=\frac{2\pi i}{\log(2)}$) and abscissa of uniform convergence $\sigma_u(\varphi)\leq0$.
		
		Let $w\in g(\mathbb{D})\setminus\{g(0)\}=\varphi(\mathbb{C}_0)\setminus\{\varphi(+\infty)\}$. The periodicity implies that 
		\begin{equation}\label{32}
		\left[\frac{2T}{p}\right]\sum\limits_{\substack{s\in\varphi^{-1}(\{w\})\\
				0\leq \Im s<p\\
				\sigma<\Re s<\infty}}\left(\Re s\right)^a\leq \sum\limits_{\substack{s\in\varphi^{-1}(\{w\})\\
				|\Im s|<T\\
				\sigma<\Re s<\infty}}\left(\Re s\right)^a\leq \left(\left[\frac{2T}{p}\right]+1\right)\sum\limits_{\substack{s\in\varphi^{-1}(\{w\})\\
				0\leq \Im s<p\\
				\sigma<\Re s<\infty}}\left(\Re s\right)^a,
		\end{equation}
		where $[x]$ is the integer part of the real number $x$.
		Note that
		\begin{align*}
		\left(\log 2\right)^{a}\sum\limits_{\substack{s\in\varphi^{-1}(\{w\})\\
				|\Im s|<p\\
				\Re s>0}}\left(\Re s\right)^a&=\sum\limits_{\substack{z\in g^{-1}(\{w\})\\
				|z|<1}}\left(\log\frac{1}{|z|}\right)^{a}.
		\end{align*}
		Writing $w= e^{-b} e^{i\theta},$ where $b>0$ and $\theta\in[0,2\pi)$, so that
		\begin{equation*}
		g^{-1}(\{w\})=\left\{z_n=\frac{1-b+i(\theta +2\pi n)}{i(\theta +2\pi n)-b-1}:n\in\mathbb{Z}\right\}.
		\end{equation*}
		We thus have
		\begin{align*}
		\sum\limits_{\substack{s\in\varphi^{-1}(\{w\})\\
				0\leq \Im s<p\\
				\Re s>0}}\left(\Re s\right)^a& \approx \sum_{n\in\mathbb{Z}}\left(1-|z_n|^2\right)^{a}\\
		&= \sum_{n\in\mathbb{Z}}\left(\frac{4b}{(1+b)^2+(\theta+2n\pi)^2}\right)^{a}.
		\end{align*}
		This shows that $M_{\varphi,a}(w)=\infty$ for all $w\in\mathbb{D}$ and $a\leq\frac{1}{2}$.
	\end{example}
\subsection{Weighted mean counting functions as integrals}
The purpose of this subsection is to replace the limiting processes in the definition of $M_{\varphi, a}$ with integration.  For $a \geq 1$, this allows us to show that it is almost always possible to directly take $\sigma = 0$ in the definition of the weighted mean counting function.
\begin{lemma}\label{invar}
Let $a\in\mathbb{R}$ and let $\varphi$ be a Dirichlet series such that $\sigma_u(\varphi)\leq 0$ and $\varphi(+\infty)\neq w$. Then, for every $\sigma>0$, the weighted mean counting function is invariant under vertical limits, that is,
$$M_{\varphi,a}(w,\sigma)=M_{\varphi_\chi,a}(w,\sigma),\qquad \chi\in\mathbb{T}^\infty.$$
\end{lemma}
\begin{proof}
The statement holds for the Jessen function, see \cite[Satz A]{JES33} or \cite[Lemma~4.1]{BP21}, 
\begin{equation*}
\mathcal{J}_{\varphi-w}(\sigma)=\mathcal{J}_{\varphi_\chi-w}(\sigma), \qquad \chi\in\mathbb{T}^\infty.
\end{equation*}
Thus, for the unweighted counting function, we have that $$M_{\varphi_\chi,0}(w,\sigma)=-\mathcal{J}'_{\varphi_\chi-w}(\sigma^+)=-\mathcal{J}'_{\varphi-w}(\sigma^+)=M_{\varphi,0}(w,\sigma).$$
By Theorem \ref{1.1} it follows that every weighted mean counting function $M_{\varphi_\chi,a}(w,\sigma)$ is invariant under vertical limits.
\end{proof}
Of course we may let $\sigma\rightarrow0^+$ to obtain that $M_{\varphi,a}(w)=M_{\varphi_\chi,a}(w)$ for every $\chi \in \mathbb{T}^\infty$. 
\begin{theorem}\label{average}
Let $\varphi$ be a Dirichlet series such that $\sigma_u(\varphi)\leq 0$ and $\varphi(+\infty)\neq w$. Then, for every $a\in\mathbb{R}$ the weighted mean counting function can be written as \begin{equation}\label{aver}
M_{\varphi,a}(w)=\int\limits_{\mathbb{T}^\infty}M_{\varphi_\chi,a}(w,0,1) \, dm_\infty(\chi).
\end{equation}
\end{theorem}
\begin{proof}
For fixed $\sigma>0$, almost periodicity and Hurwitz's theorem imply that $M_{\varphi_\chi,a}(w,\sigma,1)$ is uniformly bounded in $\chi \in \mathbb{T}^\infty$, cf. \cite[Theorem~3]{JT45}. Thus, we can apply the Birkhoff--Khinchin theorem: for almost every character $\chi' \in \mathbb{T}^\infty$, it holds that 
\begin{multline*}
\int\limits_{\mathbb{T}^\infty}M_{\varphi_\chi,a}(w,\sigma,1) \, dm_\infty(\chi)=\lim\limits_{T\rightarrow\infty}\frac{1}{2T}\int\limits_{-T}^TM_{\varphi_{n^{-it}\chi'},a}(w,\sigma,1) \, dt\\
=\lim\limits_{T\rightarrow\infty}\frac{\pi}{2T}\int\limits_{-T}^T\sum\limits_{\substack{s\in\varphi_{n^{-it}\chi'}^{-1}(\{w\})\\
		|\Im s|<1\\
	\Re s>\sigma}}\left(\Re s\right)^{a} \, dt
=\lim\limits_{T\rightarrow\infty}\frac{\pi}{2T}\int\limits_{-T}^T\sum\limits_{\substack{s\in\varphi_{\chi'}^{-1}(\{w\})\\
			-1+t<\Im s<1+t\\
			\Re s>\sigma}}\left(\Re s\right)^{a}dt.
\end{multline*}
Interchanging the order of integration and summation yields that 
\begin{equation*}
\int\limits_{\mathbb{T}^\infty}M_{\varphi_\chi,a}(w,\sigma,1) \, dm_\infty(\chi)=\lim\limits_{T\rightarrow\infty}\frac{\pi}{2T}\sum\limits_{\substack{s\in\varphi_{\chi'}^{-1}(\{w\})\\
		|\Im s|<1+T\\
		\Re s>\sigma}}\left(\Re s\right)^{a}\int\limits_{\Im s-1}^{\Im s+1}dt
=M_{\varphi_{\chi'},a}(w,\sigma).
\end{equation*}
Applying Lemma \ref{invar} and then letting $\sigma\rightarrow0^+$ with the monotone convergence theorem, we obtain \eqref{aver}.
\end{proof}
From this argument we are also able to give a partial solution to \cite[Problem~1]{BP21}.
\begin{theorem}
Let $\varphi$ be a Dirichlet series such that $\sigma_u(\varphi)\leq 0$ and $\varphi(+\infty)\neq w$. Then, for $a\geq 1$ and almost every $\chi\in\mathbb{T}^\infty$,
\begin{equation}\label{prob1}
M_{\varphi,a}(w)=\lim_{T\rightarrow \infty}\frac{\pi}{T}\sum\limits_{\substack{s\in\varphi_\chi^{-1}(\{w\})\\
		|\Im s|<T\\
		\Re s>0}}\left(\Re s\right)^a.
\end{equation}
\end{theorem}
\begin{proof}
By \cite[Lemma~2.4]{BP21}, applied in conjunction with Lemma \ref{3.4} for $a > 1$,  we have that $M_{\varphi_\chi,a}(w,0,T_0)\in L^\infty(\mathbb{T}^\infty)$ for all sufficiently large $T_0 > 0$.  Applying the Birkhoff--Khinchin theorem as in the proof of Theorem~\ref{average}, it holds for almost every $\chi'\in\mathbb{T}^\infty$ that
\begin{equation*}
 \int\limits_{\mathbb{T}^\infty}M_{\varphi_\chi,a}(w,0,T_0) \, dm_\infty(\chi)=\lim\limits_{T\rightarrow\infty}\frac{\pi}{T}\sum\limits_{\substack{s\in\varphi_{\chi'}^{-1}(\{w\})\\
		|\Im s|<T\\
		\Re s>0}}\left(\Re s\right)^a.
\end{equation*}
However, exactly as in Theorem \ref{average}, we also have that 
\begin{equation*}
	M_{\varphi_{\chi'}, a}(w) = M_{\varphi, a}(w) =  \int\limits_{\mathbb{T}^\infty}M_{\varphi_\chi,a}(w,0,T_0) \, dm_\infty(\chi). \qedhere
	\end{equation*}
\end{proof} 
	\subsection{The submean value property}
	Let $\Omega$ be an open subset of $\mathbb{C}$. We say that a function $u:\Omega\rightarrow[-\infty,\infty)$ satisfies the submean value property if for every disk $\overline{D(w,r)}\subset\Omega$
	\begin{equation*}
	u(w)\leq \frac{1}{|D(w,r)|}\int\limits_{D(w,r)}u(z)dA(z),
	\end{equation*}
	where $|D(w,r)| = \pi r^2$ is the area of the disk.
	
	Shapiro \cite[Section 4]{SHAP87} proved that for every holomorphic self-map of the unit disk $\phi$, the Nevanlinna counting function $N_\phi$ satisfies the submean value property in $\mathbb{D}\setminus\{\phi(0)\}$.
	Kellay and Lefevre \cite[Lemma 2.3]{KL12} proved that for $\alpha\in(0,1)$, the generalized Nevanlinna counting function  $N_{\phi,\alpha}$ satisfies the submean value property in $\mathbb{D}\setminus\{\phi(0)\}$. In fact, this result follows directly from the submean value property of the classical Nevanlinna counting function and the following formula due to Aleman.
	\begin{theorem}[\cite{AL92}]\label{4.8}
		Let $0<\alpha<1$ and $\phi:\mathbb{D}\rightarrow\mathbb{D}$ be holomorphic and non-constant. Then
		\begin{equation*}
		N_{\phi,\alpha}(w)=-\frac{1}{2}\int\limits_{\mathbb{D}}\Delta \omega_a(z)N_{\phi\circ \tau_z}(w) \, dA(z), \qquad w\in\mathbb{D},
		\end{equation*}
		where $\omega_a(z)=\left(1-|z|^2\right)^\alpha$ and $\tau_z(w)=\frac{z-w}{1-\overline{z}w}$.
	\end{theorem}
	In the Hardy space case the mean counting function $M_{\varphi,1}$  satisfies the submean value property \cite[Lemma 6.5]{BP21}, for every Dirichlet series $\varphi$ that belongs to the Nevanlinna class. For periodic symbols $\varphi(s)=g(2^{-s})$, where $g$ is a holomorphic self-map of the unit disk, we also know that $M_{\varphi,a}(w)$ satisfies the submean value property for all $a\in(0,1)$, by an application of Theorem~\ref{4.8}.
	
	This subsection is devoted to proving the following.
	\begin{theorem}\label{4.10}
		Let $\varphi$ be a Dirichlet series with $\sigma_u(\varphi)\leq0$. Then, for every positive $a > 0$, there exists a constant $C=C(a)>0$ such that
		\begin{equation}\label{34}
		M_{\varphi,a}(w)\leq \frac{C}{|D(w,r)|}\int\limits_{D(w,r)}M_{\varphi,a}(z) \, dA(z),
		\end{equation}
		for every disk $D(w,r)$ that does not contain $\varphi(+\infty)$.
	\end{theorem}
	
	For $a=0$ the (unweighted) counting function does not satisfy \eqref{34}, as can be seen from the following example.
	\begin{example}
		Let $\varphi_\rho(s)=g_\rho(2^{-s})$, where $g_\rho$ is a Riemann map from the unit disk onto the domain
		$$\Omega_\rho=D(0,r)\bigcup \left\{x+iy:x\in[0,2r),\, y\in(-r/\rho, r/\rho)\right\},$$
		where $r<\frac{1}{3}$ is fixed. Then for $w=\frac{119 r}{60}$, we have that
		$$ \lim_{\rho \to \infty} \frac{1}{N_{g_\rho,0}(w)}  \int\limits_{D(w, 1-2r)}N_{g_\rho,0}(z) \, dA(z) = 0,$$
		and thus by \eqref{32}, that
		$$ \lim_{\rho \to \infty} \frac{1}{M_{\varphi_\rho,0}(w)}  \int\limits_{D(w, 1-2r)} M_{\varphi_\rho,0}(z) \, dA(z) = 0.$$
		Therefore Theorem~\ref{4.10} could not be true for $a = 0$.
	\end{example}
	
	First we need the following lemma.
	\begin{lemma}\label{4.11}
		Suppose that $\Omega$ is a bounded subdomain of $\mathbb{C}$ and let $\phi:\mathbb{D}\rightarrow\Omega$ be holomorphic. Then the generalized Nevanlinna counting function $N_{\phi,\alpha}(w)$ satisfies the submean value property for $0<\alpha<1$.
	\end{lemma}
	\begin{proof}
		Let $R>0$ be such that $\overline{\Omega}\subset D(0,R)$. Then, the function $\Phi=\frac{\phi}{R}$ is a holomorphic self-map of the unit disk. We observe that for every $w\in\Omega\setminus\{\phi(0)\}$,
		\begin{equation*}
		N_{\phi,\alpha}(w)=N_{\Phi,\alpha}\left(\frac{w}{R}\right).
		\end{equation*}
		Let $\overline{D(w,r)}\subset\Omega\setminus\{\phi(0)\}$. Then, by the submean value property of $N_{\Phi, \alpha}$, we have that
		\begin{equation*}
		N_{\phi,\alpha}(w) \leq \frac{R^2}{|D(w,r)|}\int\limits_{D\left(\frac{w}{R},\frac{r}{R}\right)}N_{\Phi,\alpha}(z) \, dA(z) =\frac{1}{|D(w,r)|}\int\limits_{D(w,r)}N_{\phi,\alpha}(z) \, dA(z). \qedhere
		\end{equation*}
	\end{proof}
	
	\begin{proof}[\textbf{Proof of Theorem \ref{4.10}}]
		First we consider the case $a\in (0,1]$. In the notation of the proof of Lemma \ref{4.4}, let $\Theta_{\sigma, 2T}(z)= \Theta_{2T}(z)+\sigma =2T\Theta(z)+\sigma$ be the Riemann map from the unit disk onto the half-strip 
		$$S_{\sigma, 2T}=\left\{z: \Re z>\sigma,\, |\Im z|<2T \right\},$$
		with $\Theta_{\sigma, 2T}(0)=2T+\sigma$ and $\Theta_{\sigma, 2T}'(0)>0$. We observe that 
		\begin{equation*}
		\Theta_{\sigma, 2T}^{-1}(s)=\Theta_{2T}^{-1}(s-\sigma)=\Theta^{-1}\left(\frac{s-\sigma}{2T}\right).
		\end{equation*}
		By the Koebe quarter theorem, working as in Lemma \ref{4.4},
		\begin{align*}
		1-\left|\Theta_{\sigma, 2T}^{-1}(s)\right|^2&\approx \frac{\Re s-\sigma}{2T},
		\end{align*}
		whenever  $\sigma<\Re s<T$ and $|\Im s|<T$.
		
		For $T>0$ so large that $\varphi(s)\not\in D(w,r)$ for all $\Re s\geq T$, we have that
		\begin{align*}
		\frac{\pi}{T^{a}}\sum\limits_{\substack{s\in\varphi^{-1}(\{z\})\\
				|\Im s|<T\\
				2\sigma<\Re s<\infty}}\left(\Re s\right)^a&\leq 2^{a}\frac{\pi}{T^{a}}\sum\limits_{\substack{s\in\varphi^{-1}(\{z\})\\
				|\Im s|<T\\
				\sigma<\Re s<T}}\left(\Re s-\sigma\right)^{a}\\
		&\leq C \sum\limits_{\substack{s\in\varphi^{-1}(\{z\})\\
				|\Im s|<T\\
				\Re s>\sigma}}\left(1-\left|\Theta_{\sigma, 2T}^{-1}(s)\right|^2\right)^{a} \leq C N_{\varphi\circ\Theta_{\sigma, 2T},a}(z),
		\end{align*}
		for all $z \in D(w,r)$. 
		
		Conversely, again by the Koebe quarter theorem, there exists an absolute constant $C>0$ such that
		$$1-\left|\Theta^{-1}(s)\right|^2\leq C \Re s,$$
		for $0<\Re s<\frac{1}{2}$ and $|\Im s|<1$. That is,
		\begin{align*}
		1-\left|\Theta_{\sigma, 2T}^{-1}(s)\right|^2&\leq C \frac{\Re s-\sigma}{2T},
		\end{align*}
		whenever $\sigma<\Re s<T$ and $|\Im s|<2T$. Thus, for $z\in D(w,r)$
		\begin{equation*}
		N_{\varphi\circ\Theta_{\sigma, 2T},a}(z)\leq C \sum\limits_{\substack{s\in\varphi^{-1}(\{z\})\\
				|\Im s|<2T\\
				\Re s>\sigma}}\left(1-\left|\Theta_{\sigma, 2T}^{-1}(s)\right|^2\right)^{a} \leq C\frac{\pi}{T^{a}}\sum\limits_{\substack{s\in\varphi^{-1}(\{z\})\\
				|\Im s|<2T\\
				\Re s>\sigma}}\left(\Re s\right)^{a}.
		\end{equation*}
	
		In summary, we have shown that for all sufficiently large $T>0$ and $z \in D(w, r)$,
		\begin{equation*}
			M_{\varphi,a}(z,2\sigma,T)\leq C_1 T^{a-1}N_{\varphi\circ\Theta_{\sigma, 2T},a}(z)\leq C_2M_{\varphi,a}(z,\sigma,2T).
		\end{equation*}
		Since  $N_{\varphi\circ\Theta_{\sigma, 2T},a}$ satisfies the submean value property by Lemma \ref{4.11}, we conclude that 
		\begin{equation*}
		M_{\varphi,a}(w,2\sigma,T)\leq  \frac{C}{|D(w,r)|}\int\limits_{D(w,r)}M_{\varphi,a}(z,\sigma,2T)dA(z).
		\end{equation*}
		By Theorem~\ref{1.1} and \eqref{28} we can apply the dominated convergence theorem to let $T\rightarrow \infty$, and then let $\sigma\rightarrow0^+$ with the monotone convergence theorem, to obtain the desired property,
		$$M_{\varphi,a}(w)\leq  \frac{C}{|D(w,r)|}\int\limits_{D(w,r)}M_{\varphi,a}(z)dA(z),$$
		for a constant $C>0$ that depends only on $a\in(0,1]$.
		
		We assume now that  $a > 1$. By Tonelli's theorem we have, for every $T>0,\,\sigma>0$, and $w\neq\varphi(+\infty)$, that
		\begin{align*}
		(a-1)\int\limits_{\sigma}^{\infty}t^{a-2}M_{\varphi,1}(w,t, T) \, dt
		&=(a-1)\frac{\pi}{T}\sum\limits_{\substack{s\in\varphi^{-1}(\{w\}) \\
				|\Im s|<T\\
				\sigma<\Re s<\infty}}\Re s\int\limits_{\sigma}^{\Re s}t^{a-2}dt \\ &=M_{\varphi,a}(w,\sigma, T)-\sigma^{a-1}M_{\varphi,1}(w,\sigma, T).
		\end{align*}
		Applying the submean value property for $a = 1$, we thus find that 
		\begin{equation*}
		M_{\varphi,a }(w,2\sigma, T) \leq  \frac{C}{|D(w,r)|}\int\limits_{D(w,r)}M_{\varphi,a}(z,\sigma, 2T) \, dA(z).
		\end{equation*}
		This concludes the proof, by letting $T\rightarrow\infty$ and then $\sigma\rightarrow0^+$ in the same way as before.
	\end{proof}

	\section{Composition Operators}
	\subsection{Reproducing kernels}
	To obtain necessary conditions for a composition operator to be compact, we will make use of reproducing kernels. 
	The reproducing kernel $k_{s,a}$ of $\mathcal{D}_a$ at a point $s\in\mathbb{C}_{\frac{1}{2}}$ is given by the equation
	\begin{equation*}
	k_{s,a}(w)=1+\sum\limits_{n\geq2} \frac{1}{\left(\log(n)\right)^a} \frac{1}{n^{\overline{s}+w}}, \qquad w \in \mathbb{C}_{1/2}.
	\end{equation*}
	For fixed $a<1$, we have that
	\begin{equation*}
	\norm{k_{s,a}}^2_{a}=1+\sum\limits_{n\geq2}\frac{1}{n^{2\Re s}\left(\log(n)\right)^a}\approx \frac{1}{\left(2\Re s-1\right)^{1-a}},
	\end{equation*}
	as $\Re s\rightarrow \frac{1}{2}^+.$ We will also require slightly more detailed information about the behavior of the reproducing kernel, cf. \cite[Lemma~3.1]{OS08}.
	\begin{lemma}\label{5.1}
		Let $a \leq 1$ and $J_a(w)=\sum\limits_{n\geq1}\frac{\left(\log(n)\right)^{1-a}}{n^w}$ for $\Re w>1$. Then there exists a holomorphic function $E_a$ on $\mathbb{C}_0$ such that 
		\begin{equation*}
		J_a(w)=\frac{\Gamma(2-a)}{\left(w-1\right)^{2-a}}+ E_a(w), \qquad w\in\mathbb{C}_1.
		\end{equation*}
	\end{lemma}
	\begin{proof}
		We consider the summatory function
		$$A(x)=\sum_{n\leq x}\left(\log(n)\right)^{1-a}, \qquad x \geq 1.$$
		Summation by parts yields that
		\begin{equation}
		J_a(w)=w\sum_{n\geq 1}A(n)\int\limits_n^{n+1}x^{-w-1}dx=w\int\limits_1^\infty A(x)x^{-w-1}dx.\label{37}
		\end{equation}
		We observe that
		\begin{equation*}
		A(n)\leq \int\limits_1^{n+1}\left(\log(t)\right)^{1-a}dt\leq A(n+1).
		\end{equation*}
		Thus, with $g_a(x) := A(x)-\int\limits_1^{x}\left(\log(t)\right)^{1-a}dt$, we have that
		\begin{align*}
		\left|g_a(x)\right|&\leq \left|A(\lfloor x \rfloor)-\int\limits_1^{\lfloor x \rfloor}\left(\log(t)\right)^{1-a}dt\right|+\int\limits_{\lfloor x \rfloor}^{x}\left(\log(t)\right)^{1-a}dt\nonumber\\
		&\leq 2\left(\log(x+1)\right)^{1-a}.
		\end{align*}
		Therefore the function $E_a(w):=w\int\limits_1^\infty g_a(x)x^{-w-1}dx$ is holomorphic on $\mathbb{C}_0$, so that \eqref{37} can be written in the following form,
		\begin{equation*}
		J_a(w)=w\int\limits_1^\infty \int\limits_1^{x}\left(\log(t)\right)^{1-a}x^{-w-1}dtdx+E_a(w).
		\end{equation*}
		We can compute the integral by a change of variables,
		\begin{align*}
		w\int\limits_1^\infty \int\limits_1^{x}\left(\log(t)\right)^{1-a}x^{-w-1}dtdx&=w\int\limits_0^\infty e^{-wr}\int\limits_0^r u^{1-a}e^ududr\\
		&=\int\limits_0^\infty e^{-(w-1)u}u^{1-a}du  =\frac{1}{\left(w-1\right)^{2-a}}\oint\limits_{\Lambda_w}z^{1-a}e^{-z}dz,
		\end{align*}
		where $\Lambda_w=\{t(w-1) \,: \, t\geq 0\}$. Applying Cauchy's theorem to shift the path of integration to the positive real axis, we see that
		\begin{equation*}
		w\int\limits_1^\infty \int\limits_1^{x}\left(\log(t)\right)^{1-a}x^{-w-1}dtdx=\frac{\Gamma(2-a)}{\left(w-1\right)^{2-a}}. \qedhere
		\end{equation*}
	\end{proof}
	\subsection{The Stanton formula}
	The proof of the analogue of the Stanton formula for the weighted spaces $\mathcal{D}_a$, $a \leq 1$, relies on the work of \cite{BP21} and a generalized version of the dominated convergence theorem.
	\begin{proof}[\textbf{Proof of Theorem \ref{1.2}}]
		As explained on \cite[p. 10]{BP21}, it is a consequence of Bohr's theorem that for any  $f\in\mathcal{D}_a$, the abscissa of uniform convergence satisfies $\sigma_u(f\circ\varphi)\leq 0$. By making a non-injective change of variables in the Littlewood--Paley formula \eqref{1}, we thus have that
		$$\norm{C_\varphi(f)}_a^2=|f(\varphi(+\infty))|^2+\frac{2^{1-a}}{\Gamma(2-a)\pi}\lim\limits_{\sigma_0\rightarrow0^+}\lim\limits_{T\rightarrow+\infty}\int\limits_{\mathbb{C}_{\frac{1}{2}}}|f'(w)|^2M_{\varphi,1-a}(w,\sigma_0,T)dA(w).$$
		We extract the following equation from the proof of  \cite[Theorem 1.3]{BP21}, 
		\begin{multline*}
		\lim_{T\rightarrow+\infty}\int\limits_{\mathbb{C}_{\frac{1}{2}}}|f'(w)|^2\left(M_{\varphi,1}(w,\sigma_0,T)-M_{\varphi,1}(w,\sigma_1,T)\right)dA(w)\\ =\int\limits_{\mathbb{C}_{\frac{1}{2}}}|f'(w)|^2\left(M_{\varphi,1}(w,\sigma_0)-M_{\varphi,1}(w,\sigma_1)\right)dA(w).
		\end{multline*}
		Since both $M_{\varphi, 1 - a}(w, \sigma, T)$ and $M_{\varphi, 1}(w, \sigma, T)$ converge pointwise for $\sigma > 0$, and additionally, since for $\sigma_0<\Re s<\sigma_1$ there is a constant $C > 0$ such that $\left(\Re s\right)^{1-a}\leq C\Re s$, the generalized dominated convergence theorem \cite[Sec. 2, Ex. 20]{FOL99} thus yields that
		\begin{multline*}
		\lim_{T\rightarrow+\infty}\frac{\pi}{T}\int\limits_{-T}^T\int\limits_{\sigma_0}^{\sigma_1}|(f\circ\varphi(\sigma+it))'|^2\sigma^{1-a}d\sigma dt\\ =\int\limits_{\mathbb{C}_{\frac{1}{2}}}|f'(w)|^2\left(M_{\varphi,1-a}(w,\sigma_0)-M_{\varphi,1-a}(w,\sigma_1)\right)dA(w).
		\end{multline*}
		By the monotone convergence theorem, letting $\sigma_0\rightarrow0^+$ and then $\sigma_1\rightarrow+\infty$, we conclude that
		\begin{equation*}
		\norm{C_\varphi(f)}_a^2=|f(\varphi(+\infty))|^2+\frac{2^{1-a}}{\Gamma(2-a)\pi}\int\limits_{\mathbb{C}_{\frac{1}{2}}}|f'(w)|^2M_{\varphi,1-a}(w)dA(w). \qedhere
		\end{equation*}
	\end{proof}
	When $a \leq 0$, Theorem~\ref{1.2}, the boundedness of the composition operator $C_\varphi \colon \mathcal{D}_a \to \mathcal{D}_a$, and Theorem \ref{4.10} allow us to deduce that the weighted counting function is finite.
	\begin{corollary}
		Suppose that $\varphi \in \mathfrak{G}_0$ and that $a \leq 0$. Then $M_{\varphi,1-a}(w)$ is finite for every $w\neq\varphi(+\infty)$.
	\end{corollary} 
	
	\subsection{ Compact composition operators on Bergman spaces of Dirichlet series}
	To prove the sufficiency part of characterization we will make use of the following Littlewood-type inequality from \cite[Theorem~1.1]{BP21}.
	\begin{theorem}[\cite{BP21}]\label{5.2}
		For every $\varphi\in\mathfrak{G}_0$ the mean counting function $M_{\varphi,1}$ satisfies the estimate
		\begin{equation*}
		M_{\varphi,1}(w)\leq\log\left|\frac{\overline{w}+\varphi(+\infty)-1}{w-\varphi(+\infty)}\right|,
		\end{equation*}
		where $w\in\mathbb{C}_{\frac{1}{2}}\setminus\{\varphi(+\infty)\}$.
	\end{theorem}
	The Schwarz lemma for Dirichlet series, Lemma~\ref{3.4}, allows us to extend this estimate to the Bergman case.
	\begin{proposition}\label{5.3}
		Let $\varphi\in\mathfrak{G}_0$ and $a\geq0$. Then for every $\delta>0$ there exists a constant $C(\varphi, \delta, a)>0$ such that
		\begin{equation*}
		M_{\varphi, 1+a}(w)\leq C(\varphi, \delta, a)\left(\Re w-\frac{1}{2}\right)^{1+a}\frac{\Re\varphi(+\infty)
			-\frac{1}{2}}{\left|w-\varphi(+\infty)\right|^2}
		\end{equation*}
		for every $w\notin D(\varphi(+\infty),\delta).$
	\end{proposition}
	\begin{proof}
		Theorem \ref{5.2} and the inequality $\log x\leq\frac{1}{2}(x^2-1),\,x>0$, together with a trivial computation, shows that for every $w\neq\varphi(+\infty)$,
		\begin{equation*}
		M_{\varphi, 1}(w)\leq 2\left(\Re w-\frac{1}{2}\right)\frac{\Re \varphi(+\infty)-\frac{1}{2}}{|w-\varphi(+\infty)|^2}.
		\end{equation*}
		Given $\delta > 0$, let $\sigma > 0$ be such that $\varphi(s)\in D(\varphi(+\infty),\delta)$ for every $\Re s>\sigma$. By Lemma \ref{3.4}, for $w\notin D(\varphi(+\infty),\delta)$, we thus have that
		\begin{align*}
		M_{\varphi,1+a}(w)&=\lim_{\sigma_0\rightarrow0^+}\lim_{T\rightarrow+\infty}\frac{\pi}{T}\sum\limits_{\substack{s\in\varphi^{-1}(\{w\})\\
				|\Im s|<T\\
				\sigma_0 <\Re s\leq \sigma}}\left(\Re s\right)^{a+1}\\
		&\leq C(\varphi) \left(1+\sigma^2\right)^a\left(\Re w-\frac{1}{2}\right)^{a}M_{\varphi,1}(w)\\
		&\leq C(\varphi,\delta, a)\left(\Re w-\frac{1}{2}\right)^{1+a}\frac{\Re \varphi(+\infty)-\frac{1}{2}}{|w-\varphi(+\infty)|^2}. \qedhere
		\end{align*}
	\end{proof}
	\begin{remark} 
		By combining a similar argument with \cite[Lem. 2.4]{BP21}, we find that
		$$M_{\varphi,1+a}(w) = O\left(\left(\log\frac{1}{|w-\varphi(+\infty)|}\right)^{1+2a}\right)$$
		as $w\rightarrow\varphi(+\infty)$. We do not expect the exponent in this inequality to be the best possible.
	\end{remark}
	Before proceeding with the proof of Theorem \ref{1.3}, we require two simple lemmas, the first of which has the same proof as \cite[Lemma 7.3]{BP21}.
	\begin{lemma}\label{5.4}
		Let $\nu\in\mathbb{C}_{\frac{1}{2}},$ $a\geq0$, and $\delta>0$. Then there exists a constant $C = C(\delta,a)>0$ such that
		\begin{equation*}
		\int\limits_{\frac{1}{2}<\Re w<\theta}|f'(w)|^2\frac{\left(\Re w-\frac{1}{2}\right)^{1+a}}{|w-\nu|^{1+\delta}}dA(w)\leq \frac{C}{\left(\Re\nu-\theta\right)^{1+\delta}}\norm{f}_{-a}^2
		\end{equation*}
		for every $f\in\mathcal{D}_{-a}$ and $\frac{1}{2}<\theta<\Re\nu$.
	\end{lemma}
	\begin{lemma}\label{5.5}
		Let $\varphi\in\mathfrak{G}_0$ and $\{f_n\}_{n\geq1}$ be a sequence in $\mathcal{D}_{-a},\,a\geq0$, that converges weakly to $0$. Then, for every $\theta>\frac{1}{2}$,
		\begin{equation}\label{44}
		\lim\limits_{n\rightarrow+\infty}\left(|f_n(\varphi(+\infty))|^2+\frac{2^{1+a}}{\Gamma(2+a)\pi}\int\limits_{\Re w\geq \theta}|f_n'(w)|^2M_{\varphi,1+a}(w)dA(w)\right)=0.
		\end{equation}
		\begin{proof}
			The point evaluation at $\nu=\varphi(+\infty)$ is bounded, and thus
			\begin{equation*}
			\lim\limits_{n\to \infty} f_n(\varphi(+\infty))=\lim\limits_{n\to \infty}\langle f_n,k_{\nu,-a}\rangle_{\mathcal{D}_{-a}}=0.
			\end{equation*}
			Next, applying Montel's theorem for $\mathcal{H}^\infty$ \cite[Lemma 18]{BAY02}, it is easy to see that the sequence of the derivatives $\{f_n'\}_{n\geq1}$ converges uniformly to $0$ in $\mathbb{C}_{\theta}$. The dominated convergence theorem, which can be applied in light of Lemma \ref{3.3}, thus gives us \eqref{44}.
		\end{proof}
	\end{lemma}
	We are ready to prove Theorem~\ref{1.3}.
	\begin{proof}[\textbf{Proof of Theorem \ref{1.3}}]
		We first assume that the operator $C_\varphi$ is compact on $\mathcal{D}_{-a}$. Suppose $\{s_n\}_{n\geq1}\subset \mathbb{C}_{\frac{1}{2}}$ is an arbitrary sequence such that $\Re s_n\rightarrow\frac{1}{2}$. We observe that the induced sequence of normalized reproducing kernels $\{K_{s_n,-a}\}_{n\geq1
		}$ converges weakly to $0$, as $n\rightarrow\infty$, and therefore 
		\begin{equation}\label{45}
		\lim\limits_{n\rightarrow+\infty}\norm{C_\varphi(K_{s_n,-a})}_{-a}=0.
		\end{equation}
		Without loss of generality we can assume that for every $n\geq 1$,
		$$\Re s_n<\frac{2\Re\varphi(+\infty)+\frac{1}{2}}{3},$$
		so that the disks $D(s_n,r_n)$, $r_n=\frac{\Re s_n-\frac{1}{2}}{2}$, do not contain $\varphi(+\infty)$.
		
		By Lemma \ref{5.1} there exists a constant $C=C(a)>0$ such that
		\begin{equation*}
		\left|K_{s_n,a}'(w)\right|^2\geq C\left(\Re s_n-\frac{1}{2}\right)^{-a-3}
		\end{equation*}
		whenever $w\in D(s_n,r_n)$. Therefore
		\begin{align*}
		\norm{C_\varphi(K_{s_n,a})}^2&\geq C(a)\int\limits_{\mathbb{C}_{\frac{1}{2}}}|K_{s_n,a}'(w)|^2M_{\varphi,1+a}(w)dA(w)\\
		&\geq C(a)\left(\Re s_n-\frac{1}{2}\right)^{-a-3}\int\limits_{D(s_n,r_n)}M_{\varphi,1+a}(w)dA(w).
		\end{align*}
		The submean value property of Theorem \ref{4.10} therefore yields that
		\begin{equation*}
		M_{\varphi,a+1}(s_n)\left(\Re s_n-\frac{1}{2}\right)^{-a-1}\leq C(a) \norm{C_\varphi(K_{s_n,a})}^2.
		\end{equation*}
		We conclude, by \eqref{45}, that
		\begin{equation*}
		\lim_{\Re w\rightarrow\frac{1}{2}}\frac{M_{\varphi,1+a}(w)}{\left( \Re w-\frac{1}{2}\right)^{1+a}}=0.
		\end{equation*}
		
		Conversely, we suppose that \eqref{5} holds and argue as in the proof of \cite[Theorem 1.4]{BP21}. Let $\{f_n\}_{n\geq1}$ be a sequence in $\mathcal{D}_{-a}$ that converges weakly to $0$, such that $\norm{f_n}_{-a}\leq 1$ for all $n\geq 1$.
		Fix $\delta \in (0,1)$. By Proposition \ref{5.3} and \eqref{5}, there exists for every $\epsilon>0$ a $\theta$,  $\frac{1}{2}<\theta<\frac{\frac{1}{2}+\Re\varphi(+\infty)}{2}$, such that
		$$M_{\varphi, 1+a}(w)\leq \epsilon \frac{\left(\Re w-\frac{1}{2}\right)^{1+a}}{\left|w-\varphi(+\infty)\right|^{1+\delta}}$$
		for all $\frac{1}{2}<\Re w<\theta$.
		This and Lemma \ref{5.4} gives us that
		\begin{align*}
		\int\limits_{\frac{1}{2}<\Re w<\theta}|f_n'(w)|^2M_{\varphi,1+a}(w)dA(w)&\leq \epsilon \int\limits_{\frac{1}{2}<\Re w<\theta}|f_n'(w)|^2\frac{\left(\Re w-\frac{1}{2}\right)^{1+a}}{\left|w-\varphi(+\infty)\right|^{1+\delta}}dA(w)\\
		&\leq \epsilon C(\varphi,\delta,a).
		\end{align*}
		Combined with Lemma \ref{5.5} and Theorem \ref{1.2} we see that $\norm{C_\varphi(f_n)}_{-a}\rightarrow0$. Thus $C_\varphi$ is compact on $\mathcal{D}_{-a}$.
	\end{proof}
	\begin{remark}
		We can extract an alternative proof for the boundedness of $C_\varphi \colon \mathcal{D}_{-a} \to \mathcal{D}_{-a}$ from the second half of the proof of Theorem~\ref{1.3}.
	\end{remark}
	
	\subsection{Composition operators on Dirichlet-type spaces}
	The proof of Theorem \ref{1.4} is completely analogous to the proof of necessity in Theorem~\ref{1.3}. We leave the details to the reader. The following example illustrates that the necessary condition \eqref{6} of Theorem \ref{1.4} is not sufficient for the composition operator to be bounded on the Dirichlet space $\mathcal{D}_a$, $a\geq \frac{1}{2}$.
	\begin{example}\label{5.6}
		For $\nu\in \mathbb{C}_{\frac{1}{2}}$, the function $g_\nu(z)=\frac{(\overline{\nu}-1)z+\nu}{1-z}$ maps the unit disk $\mathbb{D}$ onto $\mathbb{C}_\frac{1}{2}$. Consider the Dirichlet series $\varphi(s)=g_\nu(2^{-s})$. For $1/2 \leq a \leq 1$, by \eqref{32}, we have that
		\begin{align*}
		M_{\varphi,1-a}(w)=\log(2)^{2-a}\log\left|\frac{1}{g_\nu^{-1}(w)}\right|^{1-a}=\log(2)^{2-a}\left(\log\left|\frac{w+\overline{\nu}-1}{w-\nu}\right|\right)^{1-a}.
		\end{align*}
		Thus, $\varphi$ satisfies \eqref{6}. However, for sufficiently small $\varepsilon > 0$, 
		\begin{align*}
		\norm{C_\varphi(2^{-s})}_a^2&\geq C(a)\int\limits_{\frac{1}{2}}^{\frac{1}{2}+\epsilon}2^{-2\sigma}\int\limits_{-\infty}^{\infty}\left(\frac{(\sigma-\frac{1}{2})(\Re\nu-\frac{1}{2})}{|\sigma+\overline{\nu}-1+it|^2}\right)^{1-a}dtd\sigma=\infty.
		\end{align*}
	\end{example}
	
	We finish the article by noting that when $\Im \varphi$ is bounded, it is simple to establish the converse to Theorem~\ref{1.4}. Note that if $C_\varphi$ is bounded on $\mathcal{D}_a$, $0 < a \leq 1$, then Theorem~\ref{1.2} also implies that $M_{\varphi, 1-a}$ is locally integrable at $w = \varphi(+\infty)$.
	
	\begin{theorem}\label{5.7}
		Let $0 < a < 1$, and suppose that $\varphi\in\mathfrak{G}_0$ has bounded imaginary part. If the counting function $M_{\varphi, 1-a}$ is locally integrable and satisfies \eqref{6}, then $C_\varphi$ is bounded on $\mathcal{D}_a$. In addition, if we assume that $\varphi$ satisfies \eqref{7}, then $C_\varphi$ is compact on $\mathcal{D}_a$.
	\end{theorem}
	\begin{proof}
		
		We present the proof of the first part of the theorem only. Let $\delta > 0$ be small. By the hypothesis of local integrability and Lemma~\ref{3.3}, it holds that
		$$\int\limits_{ D(\varphi(+\infty),\delta)}|f'(w)|^2M_{\varphi, 1-a}(w)dA(w)\leq C\norm{f}_a^2.$$
		
		Let $T = \sup |\Im \varphi|$. The local embedding theorem for the Hardy space $\mathcal{H}^2 = \mathcal{D}_0$ \cite[Theorem 4.11]{HLS97} says that
		$$\int\limits_{-T}^T |g(1/2 + it)|^2 \, dt \leq C\|g\|_0, \qquad g \in \mathcal{D}_0.$$
		Applying this with \eqref{6}, we find that
		\begin{align*}
		\int\limits_{\mathbb{C}_{\frac{1}{2}}\setminus D(\varphi(+\infty),\delta)}|f'(w)|^2M_{\varphi, 1-a}(w)dA(w) &\leq C  \int\limits_{0}^{\infty}\int\limits_{-T}^{T}\left|f'\left(\frac{1}{2}+\sigma+it\right)\right|^2dt \, \sigma^{1-a}d\sigma\\
		&\leq C \int\limits_{0}^{\infty}\sum\limits_{n\geq 2}\frac{|a_n|^2\log(n)^2}{n^{2\sigma}}\sigma^{1-a}d\sigma
		\leq C \norm{f}_a^2.
		\end{align*}
		In light of Theorem \ref{1.2}, this shows that $C_\varphi \colon \mathcal{D}_a \to \mathcal{D}_a$ is bounded. 
	\end{proof}
	From the proof it is clear that Theorem~\ref{5.7} also holds under milder decay assumptions on $M_{\varphi, 1-a}(w)$ as $|\Im w| \to \infty$. 
	\bibliographystyle{amsplain-nodash} 
	\bibliography{ref} 
\end{document}